\documentclass[draft,reqno]{amsart}

\usepackage{amsmath}
\usepackage{amssymb}
\usepackage[all]{xy}
\usepackage{picinpar}
\usepackage{palatino}
\usepackage[usenames,dvipsnames,svgnames,table]{xcolor}
\usepackage{mathtools}

\def\eu{\mathfrak}
\def\ma{\mathbb}
\def\mc{\mathcal}

\def\p{{\eu p}_{\infty}}
\def\f{{\ma F}_q^{\ast}}
\def\F{{\ma F}_q}
\def\P{\mathcal P}
\def\pK{\eu p}
\def\pL{\eu P}
\def\fin{\hfill\qed\bigskip}
\def\ge#1{#1_{\eu{gex}}}
\def\g#1{#1_{\eu{ge}}}
\def\*#1{#1^*}
\def\cicl#1#2{k(\Lambda_{{#1}^{#2}})}

\def\lra{\longrightarrow}
\def\poly{P_1^{\alpha_1}\cdots P_r^{\alpha_r}}

\def\Ku#1#2{k\big(\sqrt[#1]{#2}\big)}
\def\kum#1#2{\F\big(\sqrt[l^#1]{#2}\big)}
\def\rad#1#2{\sqrt[l^#1]{#2}}
\def\lra{\longrightarrow}
\def\H{{\mc H}}

\newcommand{\Irr}{\operatorname{Irr}}

\newcommand{\Gal}{\operatorname{Gal}}
\newcommand{\lcm}{\operatornamewithlimits{lcm}}
\newcommand{\mcd}{\operatorname{gcd}}
\newcommand{\Id}{\operatorname{Id}}

\newcommand{\im}{\operatorname{im}}
\newcommand{\N}{\operatorname{N}}

\newcommand{\Hom}{\operatorname{Hom}}
\newcommand{\xbinom}{\genfrac(){.5pt}0}
\newcommand{\ext}{\operatorname{ext}}

\newcounter{bean}
\newcounter{2bean}

\def\l{
\begin{list}
{\rm{(\alph{bean}).-}}{\usecounter{bean}
\setlength{\labelwidth}{0.8in}
\setlength{\labelsep}{0.3cm}
\setlength{\leftmargin}{1cm}}}

\def\las{\begin{list}
	{{\rm {(\arabic{2bean})}}}{\usecounter{2bean}
\setlength{\labelwidth}{0.8in}
\setlength{\labelsep}{0.3cm}
\setlength{\leftmargin}{1cm}}}

\numberwithin{equation}{section}
\newtheorem{theorem}{Theorem}[section]
\newtheorem{proposition}[theorem]{Proposition}
\newtheorem{lemma}[theorem]{Lemma}

\newtheorem{remark}[theorem]{Remark}
\newtheorem{definition}[theorem]{Definition}
\newtheorem{corollary}[theorem]{Corollary}

\title[Extended genus fields of abelian extensions of rational function fields]
{Extended genus fields of abelian extensions of rational function fields}

\author[J.C. Hernandez]{Juan Carlos Hernandez-Bocanegra}
\address{Departamento de Control Autom\'atico\\
Centro de Investigaci\'on y de Estudios Avanzados del I.P.N.}
\email{juan.cuencame@gmail.com, juanc.hernandez@cinvestav.mx}
 
\author[G. Villa]{Gabriel Villa--Salvador}
\address{Departamento de Control Autom\'atico\\
Centro de Investigaci\'on y de Estudios Avanzados del I.P.N.}
\email{gvillasalvador@gmail.com, gvilla@ctrl.cinvestav.mx}



\date{April 26th., 2024}

\begin{document}

\begin{abstract}

In this paper we obtain the extended genus field of a finite abelian
extension of a global rational function field. 
We first study the case of a cyclic extension of prime power
degree. Next, we use that
the extended genus fields of a composite of two cyclotomic extensions
of a global rational function field is equal to the composite
of their respective extended genus fields,
to obtain our main result. This result is that
the extended genus field of a general finite abelian extension of a
global rational function field,
is given explicitly in terms of the field and of the extended genus
field of its ``cyclotomic projection''.

\end{abstract}

\maketitle

\section{Introduction}\label{S1}

The concepts of genus field and extended (or narrow) genus field
depend on the respective concepts of Hilbert Class Field (HCF) and extended
(or narrow) Hilbert Class Field. The theory of the genus goes back
to Gauss. The HCF concept is much more recent.
The first to translate the theory of Gauss on genus to ``modern terms'',
was Hilbert. Nowadays
it may be used to study the ``easy'' part of the HCF
of a finite extension of the field of rational numbers.

The first to give a definition of a genus field for number
fields, was H. Hasse, who defined the genus field of a quadratic extension
of ${\ma Q}$. Since the genus field is related with the HCF,
one natural way to study genus fields is by means of class
field theory. However, we may study genus fields of abelian extensions
of the rational field by means of Dirichlet characters.

For number fields, the definition of Hilbert and extended Hilbert Class
Field are canonically given as the maximal abelian unramified and the maximal
abelian unramified at the finite primes of the field respectively. The definition of the
genus field is not absolute, as is the HCF, but depends on an extension. A.
Fr\"ohlich gave a general definition of genus fields for any number field.
Fr\"ohlich definition is also canonical.

We are interested in global function fields. In this context, there are several
different definitions of HCF of a global field $K$ depending on which aspect
we are interested in. In this paper we study the extended genus field of a finite
abelian extension $K/k$ where $k={\ma F}_q(T)$ is a global
rational function field. Let $\p$ be infinite prime of $k$. Then we define
HCF of $K$ as the maximal unramified abelian extension $K_H$ of $K$ such
that the infinite primes of $K$ (those above $\p$) decompose fully
in $K_H$. B. Angl\`es and J.-F. Jaulent (\cite{AnJa2000}) give the
same concept by means of the id\`ele norm subgroup
corresponding to $K_H$.
They also define the extended HCF of any global field $K$ by means
of the norm subgroup of $K_{H^+}$ in the id\`ele group $J_K$ of $K$.
We use the definitions of Angl\`es and Jaulent of $K_H$ and $K_{H^+}$
to define the genus $\g K$ and the extended genus $\ge K$ fields
of $K$ with respect to the extension $K/k$ and compare our findings
with the concepts of extended genus fields given in \cite{RaRzVi2019}
and in \cite{Cle92}. 

We first study the case of a cyclic extension of $k$ of prime power
degree $l^n$. We consider four possible type of primes $l$: (1) $l=
p$, the characteristic of $k$ (Artin-Schreier-Witt case);
(2) $l\nmid q-1$, $l\neq p$; (3) $l^n|q-1$ (Kummer case),
and (4) $l^{\rho}|q-1$, $1\leq \rho<n$ and $l^n\nmid q-1$ (``semi-Kummer''
case).
Our main results are Theorems \ref{TExtendido5.19} and \ref{TEx2.2}.

The main tools used in this paper are the Carlitz theory of cyclotomic
function fields and class field theory, particularly the concepts of HCF
and genus fields developed by Angl\`es and Jaulent.

\section{Notations and general results}\label{S2}

For the general Carlitz--Hayes 
theory of cyclotomic function fields, we
refer to \cite[Ch. 12]{Vil2006} and \cite[Cap. 9]{RzeVil2017}.
For the results on genus fields of function fields we
refer to \cite{BaMoReRzVi2018,MaRzVi2013,MaRzVi2017}
 and \cite[Cap. 14]{RzeVil2017}.

We will be using the following notation.
Let $k=\F(T)$ be a global rational function field,
where $\F$ is the finite field of $q$ elements. Let
$R_T=\F[T]$ and let $R_T^+$ denote the set of the monic
irreducible elements of $R_T$. For $N\in R_T$, $\cicl N{}$ denotes
the $N$--th cyclotomic function field where $\Lambda_N$ is
the $N$--th torsion of the Carlitz module. For a $D\in R_T$ we
define $\*D:=(-1)^{\deg D} D$.

We will call a field $F$ a {\em cyclotomic function field} 
if there exists $N\in R_T$ such that $F\subseteq \cicl N{}$.

Let $N\in R_T$. The Dirichlet characters $\chi\bmod N$ are
the group homomorphisms $\chi\colon (R_T/\langle N\rangle)^*
\lra \*{\ma C}$. Given a group $X$ of Dirichlet characters
modulo $N$, the {\em field associated to $X$} is the fixed field $F=
\cicl N{}^H$ where $H=\bigcap_{\chi\in X}\ker \chi$. We say that $F$
corresponds to the group $X$ and that $X$ corresponds to $F$.
We have that $X\cong \Hom(\Gal(F/k),\*{\ma C})$. When $X$ is a
cyclic group generated by $\chi$, we have that the field associated
to $X$ is equal to $F=\cicl N{}^{\ker \chi}$ and we say that $F$
corresponds to $\chi$.

Given a cyclotomic function field $F$ with Dirichlet group characters $X$,
we have that the ramification index of $P\in R_T^+$ in $F/k$ equals
$|X_P|$ where $X_P=\{\chi_P\mid \chi\in X\}$ and $\chi_P$ is the
$P$--th component of $\chi$, see \cite{MaRzVi2013}. 
The maximum cyclotomic extension of $F$ unramified at the finite prime
divisors is the field that corresponds to $Y:=\prod_{P\in R_T^+}
X_P$ This field
is denoted as $\ge F$.

We denote the infinite prime of $k$ by $\p$. That is, $\p$ is the pole
divisor of $T$ and $1/T$ is an uniformizer for $\p$.

Given a finite extension $K/k$ and a definition of HCF $K_H$ and
of extended HCF $K_{H^+}$ of $K$, the respective genus and
extended genus field of $K$ with respect to $k$ are the
extensions $KL$ such that $L$ is the maximal abelian extension
of $k$ contained in $K_H$ and in $K_{H^+}$, respectively.

We will use both notations: $e_{P}(F|k)$ or $e_{F/k}(P)$ to denote
the ramification index of the prime $P$ of $k$ in $F$. For the
place $\p$ we use the notation $e_{\infty}(F|k)$.

When $K/k$ is a finite abelian extension, it follows from
the Kronecker--Weber Theorem that 
there exist $N\in R_T$, $n\in{\ma N}\cup\{0\}$ and $m\in{\ma N}$
such that $K\subseteq {_n\cicl N{}_m}$, where, for any $F$, $F_m:=F
{\ma F}_{q^m}(T)$, for any 
$N\in R_T$, ${_n\cicl N{}}:= L_n\cicl N{}$, 
and $L_n$ is the maximum subfield of $\cicl {1/T^n}{}$
where $\p$ is totally and wildly ramified.
Then we define 
\begin{gather}\label{EcExtended1}
E:={\mc M} K\cap \cicl N{}
\end{gather}
where ${\mc M}=L_nk_m$. Then
$\g K=\g E^{\H} K$ where $\H$ is the decomposition group of
the infinite primes in $K\g E/K$ (see \cite[Theorem 2.2]{BaMoReRzVi2018}).
The group $\H$ is also the decomposition group of the infinite
primes of $K$ in $KE/K$.

For any global function field $L$, ${\ma P}_L$ denotes the set of
all places of $L$.

 For $x\in{\ma Z}$,
$v_l(x)$ denotes the valuation of $x$ at $l$. That is, $v_l(x)=\gamma$
if $l^{\gamma}|x$ and $l^{\gamma+1}\nmid x$. We write $v_l(0)=\infty$.

\section{Basic results}\label{ExtendedS2+1}

One result on ramification of tamely ramified extensions,
used frequently, is the following theorem.

\begin{theorem}[Abhyankar's Lemma]\label{T2.1}
Let $L/K$ be a separable of global function fields. Assume that
$L=K_1K_2$ with $K\subseteq K_i\subseteq L$, $1\leq i\leq 2$. Let
$\pK$ be a prime divisor of $K$ and $\pL$ a prime divisor in $L$
above $\pK$. Let $\pL_i:=\pL\cap K_i$, $i=1,2$. If at least one of the
extensions $K_i/K$ is tamely ramified at $\pK$, then
\[
e_{L/K}(\pL|\pK)=\lcm[e_{K_1/K}(\pL_1|\pK),e_{K_2/K}(\pL_2|\pK)],
\]
where $e_{L/K}(\pL|\pK)$ denotes the ramification index. $\fin$
\end{theorem}

Next, we present some basic facts on finite cyclic groups
and we apply them to the case of a finite field.

Let $G$ be a cyclic group of order $n$, say $G=\langle
a\rangle$. Let $\Lambda_m$ the unique subgroup of $G$
or order $m$ where $m|n$. We have $\Lambda=\langle
a^{n/m}\rangle$. Let $t\in {\mathbb N}$ and let $G^t:=
\{x\in G\mid x=y^t \text{\ for some $y\in G$}\}$. We have 
$G^t=\im \varphi_t$, where $\varphi_t\colon G\to G$ is
given by $\varphi_t(x)=x^t$.

Note that if $t\in{\mathbb N}$ and $d=\gcd(t,n)$, then $G^d=
G^t$, namely, if $\alpha,\beta\in {\mathbb Z}$
are such that $\alpha t+\beta n=d$, we have
\[
G^d=G^{\alpha t+\beta n}=(G^{\alpha})^t(G^n)^{\beta}\subseteq
G^t\cdot 1=G^t.
\]
Conversely, let $t=\kappa d$ with $\kappa\in{\mathbb N}$. Then
$G^t=(G^{\kappa})^d\subseteq G^d$.

We also have that $\Lambda_d=\Lambda_t$ since if $x\in
\Lambda_t$, then $x^t=1=(x^t)^{\alpha}\cdot(x^n)^{\beta}=
x^{\alpha t+\beta n}=x^d$ so that $\Lambda_t\subseteq
\Lambda_d$. Conversely, if $t=\kappa d$ and if 
$x\in\Lambda_d$ then $1=x^t=(x^d)^{\kappa}=1^{\kappa}=1$.

Now, if $d|n$, then $G^d=\Lambda_{n/d}$ because we have
the exact sequence
\[
1\lra \Lambda_d\lra G\stackrel{\varphi_d}{\lra}G^d\lra 1,
\]
obtaining $|G^d|=\frac{|G|}{|\Lambda_d|}=\frac nd=
|\Lambda_{n/d}|$ and, if $x\in G^d$, there exists $y\in G$ such that
$x=y^d$ that implies $x^{n/d}=(y^d)^{n/d}=y^n=1$ so that
$G^d\subseteq \Lambda_{n/d}$. Thus $G^d=\Lambda_{n/d}$.

We apply the previous basic results to the multiplicative groups
of the finite field $\f$ that is a cyclic group of $q-1$ elements.

\begin{lemma}\label{LEx1.4}
Let $l$ be a prime number and $n\in{\mathbb N}$, with $l^n|
q-1$. Let $F:=\kum{n}{\beta}$ with $\beta\in \f$. Then $F=
{\mathbb F}_{q^{l^s}}$ for some $0\leq s\leq n$.
\end{lemma}

\begin{proof}
Let $\mu=\sqrt[l^n]{\beta}$. Then $\mu^{l^n}=\beta\in \f$. Set
$s$, $0\leq s\leq n$ to be the 
minimal non-negative integer such that $\mu^{l^s}=\theta\in\f$.
If $s=0$, then $\mu=\theta\in\f$ and $f(X)=X^{l^s}-\theta
=X-\theta$ is irreducible.

For any $s$, we will see that $f(X)=X^{l^s}-\theta$ is an
irreducible polynomial. We have $f(X)=X^{l^s}-\theta=
\prod_{j=1}^{l^s}(X-\zeta_{l^s}^j \mu)$, where $\zeta_m$ denotes
a primitive $m$-th root of unity. Let ${\mc G}:=\Gal
(\F(\mu)/\F)$. Let $\sigma\in{\mc G}$, $\sigma\neq
\Id$ and let $\sigma(\mu)=
\zeta_{l^s}^j\mu$, where $j=j_0l^b$, with $\mcd(j_0,l)=1$. We
choose an element $\sigma\in {\mc G}$ such that $b$ is minimal.

We have $\sigma(\mu)=\zeta_{l^s}^{j_0l^b}\mu=\zeta_{l^{s-b}}^{j_0}
\mu$. Let $i_0\in{\mathbb Z}$ be such that $j_0i_0\equiv 1
\bmod l^s$. Then $\sigma^{i_0}(\mu)=\zeta_{l^{s-b}}^{j_0i_0}(\mu)=
\zeta_{l^{s-b}}\mu$. This ${\mc G}$ is a cyclic group and of
order $l^{s-b}$ and
\[
\prod_{\varepsilon=1}^{l^{s-b}} (X-\zeta_{l^{s-b}}^{\varepsilon}\mu)
= X^{l^{s-b}}-\mu^{l^{s-b}}\in\F[X].
\]
Hence $\mu^{l^{s-b}}\in\f$. Therefore $b=0$, $|{\mc G}|=l^s$,
and $X^{l^s}-\theta=\Irr(\mu,X,\F)$ is irreducible.

It follows that $[\F(\mu):\F]=|{\mc G}|=l^s$ and $F=\F(\mu)=
\F(\sqrt[l^n]{\beta})=\F(\sqrt[l^s]{\theta})={\mathbb F}_{l^s}$.
\end{proof}

\begin{remark}\label{REx1.5}{\rm{
We have $\F(\sqrt[l^n]{\beta})={\mathbb F}_{q^{l^s}}$, where $s$
is the minimal non-negative integer such that $\mu^{l^s}\in \f$,
with $\mu=\sqrt[l^n]{\beta}$.
}}
\end{remark}

\begin{corollary}\label{CEx1.6}
With the above notations, if $m\in{\mathbb N}$, we have
$[\F(\sqrt[l^{n+m}]{\beta}):\F]=l^{s+m}$.
\end{corollary}

\begin{proof}
Set $\delta=\sqrt[l^{n+m}]{\beta}$, then $\delta^{l^m}=\mu=
\sqrt[l^n]{\beta}$, and $\mu^{l^s}=\delta^{l^{m+s}}=\theta
\in\f$. Clearly $m+s$ is minimal.
\end{proof}

We consider $\mu=\sqrt[l^n]{\beta}$, with $\beta\in \f$ and
$[\F(\sqrt[l^n]{\beta}:\F]=l^s$ for some $0\leq s\leq n$. Set
$\mu^{l^s}=\theta\in\f$, $\beta=\mu^{l^n}=(\mu^{l^s})^{l^{n-s}}
=\theta^{l^{n-s}}$. Hence $\beta\in
G^{l^{n-s}}$, where $G=\f$.

In case that there would exist $\varepsilon\in \f$ such that
$\beta={\varepsilon}^{l^{n-s+1}}$ it would imply that $\mu^{
l^n}=\varepsilon^{l^{n-s+1}}$ and thus $\mu=(\varepsilon^{
l^{n-s+1}})^{1/l^n}=\varepsilon^{l^{n-s+1-n}}=\varepsilon^{
l^{-s+1}}=\sqrt[l^{s-1}]{\varepsilon}$. Therefore $X^{l^{s-1}}
-\varepsilon$ has $\mu$ as a root. It would follow that
$[\F(\mu):\F]\leq l^{s-1}<l^s$ contrary to our hypothesis,
Therefore $\beta\in G^{l^{n-s}}\setminus G^{l^{n-s+1}}$.

Conversely, if  $\beta\in G^{l^{n-s}}\setminus G^{l^{n-s+1}}$,
then $\beta=\kappa^{l^{n-s}}$ where $\kappa$ in not an
$l$-power. Therefore
\[
\mu^{1/l^n}=(\kappa)^{l^{n-s}})^{1/l^n}=\kappa^{l^{-s}}=
\sqrt[l^s]{\kappa},
\]
hence $\F(\sqrt[l^n]{\mu})\subseteq {\mathbb F}_{q^{l^s}}$.
In case that $\F(\sqrt[l^n]{\mu})\subseteq {\mathbb F}_{
q^{l^{s-1}}}$ it would follow that $\mu^{l^{s-1}}\in\f$
contrary to our hypothesis.

We have proved:

\begin{theorem}\label{TEx1.6}
We have that $[\F(\sqrt[l^n]{\beta}):\F]=l^s$ if and only if
$\beta\in ({\f})^{l^{n-s}}\setminus (\f)^{l^{n-s+1}}$. $\fin$
\end{theorem}

A basic result on cyclic groups of prime power degree
that we need is the following.

\begin{proposition}\label{PExtended5.4.1}
Let $G$ be a cyclic group of order $l^{\tau}$ with $l$ a prime
number. Given $H_1,H_2 < G$, then $H_1\subseteq H_2$ or
$H_2\subseteq H_1$. In particular $H_1\cap H_2=H_j$ with
$j=1$ or $j=2$ and if $H_1\neq \{\Id\}$ and $H_2\neq \{\Id\}$,
then $H_1\cap H_2\neq \{\Id\}$.
\end{proposition}

\begin{proof}
By cyclicity, $G$ has a unique subgroup of each of the divisors
of $|G|=l^{\tau}$. These subgroups are $\Lambda_0\subseteq
\Lambda_1\subseteq \cdots\subseteq \Lambda_{\tau}$ with
$|\Lambda_i|=l^i$. The result follows.
\end{proof}

\section{Extended genus fields and class field theory}\label{S3-1}

First we establish the definition of {\em extended genus fields} according
to Angl\`es and Jaulent \cite{AnJa2000}.
For the terminology and notations we refer to \cite{RzeVil2023}.

We have that $k_{\infty}\cong \F\big(\big(\frac{1}{T}\big)\big)$ is the
completion of $k$ at $\p$. Let $x\in \*{k_{\infty}}$. Then $x$ is
written uniquely as
\[
x=\Big(\frac 1T\Big)^{n_x} \lambda_x\varepsilon_x
\quad \text{with} \quad n_x\in
{\ma Z},\quad \lambda_x\in \*{\F}\quad\text{and}\quad \varepsilon_x
\in U_{\infty}^{(1)},
\]
where $U_{\infty}^{(1)}=U_{\p}^{(1)}$ is the group of the one
units of $k_{\infty}$. We write $\pi_{\infty}:=1/T$, 
which is an uniformizer at $\p$.

The {\em sign function} is defined
as $\phi_{\infty}\colon \*{k_{\infty}}\lra \*{\F}$ given by $\phi_{\infty}
(x)=\lambda_x$ for $x\in \*{k_{\infty}}$. 

We have that $\phi_{\infty}$ is an epimorphism and $\ker \phi_{\infty}=
\langle\pi_{\infty}\rangle\times U_{\infty}^{(1)}$.

For a finite separable extension $L$ of $k_{\infty}$, we define
the {\em sign of $\*L$} by the morphism $\phi_L:=\phi_{\infty}\circ
\N_{L/k_{\infty}}\colon \*L\lra \*{\F}$. We have $\frac{\*L}
{\ker \phi_L}\cong A\subseteq \*\F$.

For a global function field $L$, let $\P$ be the set of places of $L$
dividing $\p$.
We define the following subgroups of the group of id\`eles $J_L$ as:
\begin{gather}\label{Ec0}
U_L:=\prod_{v|\infty}\*{L_v}\times \prod_{v\nmid\infty} 
U_{L_v}\quad\text{and}\quad
U_L^+:=\prod_{v|\infty}\ker \phi_{L_v}\times \prod_{v\nmid \infty} U_{L_v},
\end{gather}
where we denote $v\nmid\infty$ if $v\notin \P$ and $v|\infty$ if $v\in \P$.
The groups $U_L L^*$ and $U_L^+ L^*$ are open subgroups of 
$J_L$, the id\`ele group of $L$. 

\begin{definition}\label{D3.9}
Let $K/k$ be a finite abelian extension. Then the {\em Hilbert Class Field
(HCF)} $K_H$ and the {\em extended HCF} $K_{H^+}$ of $K$ are
the fields corresponding to the id\`ele
subgroups $U_K \*K$ and $U_K^+ \*K$ of $J_K$ respectively. 
By class field theory 
(\cite[Theorem 17.6.198]{RzeVil2017}), the respective 
genus $\g K$ and extended genus fields $\ge K$
with respect to the extension $K/k$, correspond
to the id\`ele subgroups
$(\N_{K/k}U_K)\*k$ and $(\N_{K/k}U_K^+)\*k$ of $J_k$ respectively.
\end{definition}

By class field theory we have
\[
\Gal(K_H/K)\cong J_K/U_K\*K.
\]
We have that $K_{H^+}/K$ is an unramified extension at the finite prime
divisors of $K$, $K_H\subseteq K_{H^+}$ and
\[
\Gal(K_{H^+}/K)\cong J_K/U_K^+\*K.
\]

Let $K/k$ be a finite abelian extension then,
with the notation given above, if $E={\mc M}K\cap \cicl N{}$,
we have $\ge K=DK$ for some subfield $\ge{(\g E^{\H})}\subseteq D
\subseteq \ge E$ for some decomposition group $\H$ (see \cite{RzeVil2023}).
In most cases we have $\H=\{\Id\}$. In this case, $\ge{(\g E^{\H})}=\ge E$
and $\ge K=\ge E K$. 

In this paper, we particularly study the case $\ge{(\g E^{\H})}\neq \ge E$.
A general result is the following.

\begin{proposition}\label{PExtended4.1}
Let $K/k$ be a finite abelian extension and let $E$ be given by 
{\rm{(\ref{EcExtended1})}}. Then, if $P_1,\ldots,P_r$ are the finite
primes of $k$ ramified in $K$ and
\[
e_{P_j}(\g E^{\H}|k)=e_{P_j}(E|k)=e_{P_j}(\g E|k),
\]
for all $1\leq j\leq r$, it follows that $(\ge{\g E^{\H})}=\ge E$.
\end{proposition}

\begin{proof}
The group of Dirichlet characters associated to $\ge E$ is 
\[
Y:=\prod_{
P\in R_T^+}X_P=\prod_{j=1}^r X_{P_j},
\]
where $X$ is the group of Dirichlet characters 
associated to $E$. Each $X_{P_j}$ is cyclic of order $e_{
P_j}(E|k)$. The field associated to $X_{P_j}$ is the
subfield of $\cicl {P_j}{}$ of degree $e_{P_j}(E|k)$ over $k$.
It follows that 
\[
[\ge E:k]=\prod_{j=1}^r e_{P_j}(E|k).
\]

Let $Z$ be the group of Dirichlet characters associated to
$\ge{(\g E^{\H})}$. The only finite primes of $k$ possibly ramified in
$\ge{(\g E^{\H})}$ are $P_1,\ldots, P_r$. The group of Dirichlet
characteres to $\ge{(\g E^{\H})}$ is $\prod_{j=1}^r Z_{P_j}$.
By hypothesis, $X_{P_j}=Z_{P_j}$ for all $1\leq j\leq r$. 
Therefore $\ge{(\g E^{\H})}=\ge E$.
\end{proof}

\begin{corollary}\label{CExtended4.2}
We have $\ge{(\g E^{\H})}\neq \ge E \iff$ there exists $1\leq j_0\leq r$
such that $e_{P_{j_0}}(\ge E^{\H}|k)<e_{P_{j_0}}(E|k)=e_{P_{j_0}}(\g
E|k)=e_{P_{j_0}}(\ge E|k)$.
\end{corollary}

\begin{proof}
It follows from Proposition \ref{PExtended4.1} and from the facts
that $\ge E/\g E$ is not ramified at any finite prime and that $\p$
is fully ramified in $\ge E/\g E$.
\end{proof}

Therefore, if $\ge{(\g E^{\H})}\neq \ge E$, or, equivalently, there exists
a finite prime $P_j$ ramified in $\ge E/\g E^{\H}$. The
prime $\p$ is fully ramified in $\ge E/\g E^{\H}$.

We have that $\H\subseteq I_{\infty}(\g E/k)$, the inertia group
of $\p$ in the extension $\g E/k$. 

It follows that $\Gal(\ge E/\g E^{\H})\subseteq I_{\infty}(\ge E/k)$
and we have that $I_{\infty}(\ge E/k)$ is a cyclic group of order
a power of $l$. Set $I:=I_{\infty}(\ge E/k)$, an $l$--cyclic group.
Furthermore $|I|=e_{\infty}(\ge E|k)|q-1$.

Using Proposition \ref{PExtended5.4.1}, we obtain:

\begin{proposition}\label{PExtended5.4.2}
$\ge E=\g E$.
\end{proposition}

\begin{proof}
Let $G:=I=I_{\infty}(\ge E/k)$, $H_1:=I_{P_i}(\ge E/\g E^{\H})$ and
$H_2:= \Gal(\ge E/\g E)$. By hypothesis, we have $H_1\neq\{\Id\}$.
Set $\Phi:=H_1\cap H_2$ and $F:=\ge E^{\Phi}$.
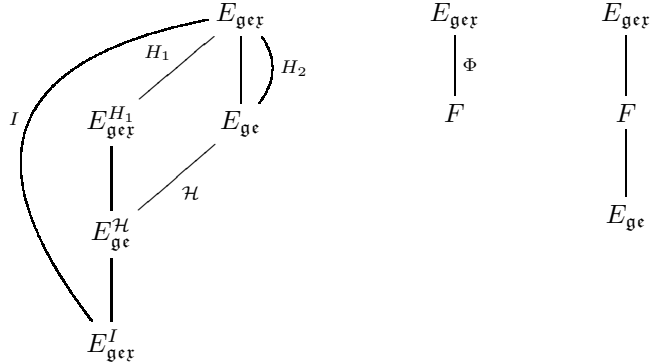
\[
\xymatrix{
& \ge E\ar@{-}[d]\ar@/^1pc/@{-}[d]^{H_2}\ar@{-}[dl]_{H_1}\\
\ge E^{H_1}\ar@{-}[d]&\g E\ar@{-}[dl]^{\H}\\ \g E^{\H}\ar@{-}[d]\\
\ge E^I\ar@{-}@/^5pc/@{-}[uuur]^I
}
\qquad\qquad
\xymatrix{
\ge E\ar@{-}[d]^{\Phi}\\F
}
\qquad\qquad
\xymatrix{
\ge E\ar@{-}[d]\\F\ar@{-}[d]\\ \g E
}
\]
Then $P_i$ is fully ramified in $\ge E/F$ and $\g E\subseteq F$.
Therefore $P_i$ is fully ramified and non-ramified in $\ge E/F$.
Hence $F=\ge E$ and $H_1\cap H_2=\{\Id\}$. Finally, it follows that
$H_2=\{\Id\}$ and that $\ge E=\g E$.
\end{proof}

In the rest of the section we will focus in the cases $\ge{(\g E^{\H})}
\neq \ge E$.

\begin{definition}\label{D3.10}
For a finite abelian extension $K/k$ we define $K^{\ext}:=\ge E K$,
where $E$ is given by {\rm{(\ref{EcExtended1})}}.
\end{definition}

In this paper we will proof that always
\[
K^{\ext}=\ge K=\ge E K,
\]
where $K/k$ is any finite abelian extension.

In \cite{RaRzVi2019} we define the extended genus field of a 
finite abelian extension $K/k$ as $K^{\ext}$. In general we have
$\ge K\subseteq K^{\ext}$.

We recall the following theorem from class field theory.

\begin{theorem}\label{TEx1.2} 
Let $F$ be a global function field. Let $N/F$ be a
finite abelian extension and let $B<C_F$ be the the subgroup of
the id\`ele class group of $F$ corresponding to $N$. Then, if $\F$ is
the field of constants of $F$, ${\mathbb F}_{q^{\kappa}}$ is the
field of constants of $N$, where
\[
\kappa:=\min\{\sigma\in{\mathbb
N}\mid \text{there exists $\tilde{\vec \alpha}\in B$ such that $\deg 
\tilde{\vec \alpha}=\sigma$}\}.
\]
\end{theorem}

\begin{proof}
See \cite[Teorema 17.6.192]{RzeVil2017}.
\end{proof}

\begin{remark}\label{REx1.3}{\rm{Note that
if $F$ is any global function field, then the field of the
constants of $\ge F$ and of $F_{H^+}$ is the same.
Therefore, in the especial case 
$\ge{(\g E^{\H})}\neq \ge E$, if we obtain that if the
field of constants of $\ge E K$ is contained in the one of
$K_{H^+}$ then, because $\ge E K=\g E K$ is an
extension of constants of $\g K$, say $\ge E K=\g K
{\ma F}_{q^{\mu}}$ and $\g K\subseteq K_H\subseteq
K_{H^+}$ and ${\ma F}_{q^{\mu}}\subseteq K_{H^+}$,
it follows that $\ge E K\subseteq K_{H^+}$ and
therefore $\ge E K\subseteq \ge K$. Since
$\ge K\subseteq \ge EK$ (see \cite{RzeVil2023}),
we obtain $\ge E K=\ge K$. 
}}
\end{remark}

Consider a finite abelian extension $K/k$.
The id\`ele class subgroup $B$ of the id\`ele class group
$C_K$, the id\`ele class group of $K$, associated to $K_{H^+}$ is
\begin{gather*}
B=U_K^+\*K/\*K=\Big(\prod_{\pL| \infty}\ker \phi_{\pL}\times
\prod_{\pL\nmid \infty}U_{\pL}\Big)\*K/\*K,
\intertext{where} 
U_{\pL}:=U_{K_{\pL}}, \quad \ker\phi_{\pL}:=\ker\phi_{K_{\pL}},
\quad \text{and} \quad
\phi_{\pL}=\phi_{\infty}\circ \N_{K_{\pL}|k_{\infty}}.
\end{gather*}
We denote $\N_{\pL}:=\N_{K_{\pL}|k_{\infty}}$. 

On the one hand, if $\vec\alpha\in U_K^+$, then $\vec\alpha=
\big(\alpha_{\pL}\big)_{\pL}$, $\alpha_{\pL}\in U_{\pL}$ for
$\pL\nmid \infty$, so that $\deg_{\pL} \alpha_{\pL}=0$
for $\pL\nmid \infty$. On the other hand, if $\pL_1$ and
$\pL_2$ are two infinite primes, $\N_{\pL_1}\*{K_{\pL_1}}
=\N_{\pL_2}\*{K_{\pL_2}}$. It follows that if we fix an
infinite prime $\pL$ of $K$, then

\begin{lemma}\label{LEx1.6+1}
We have
\begin{align*}
\kappa:&=\min\{\sigma\in{\mathbb N}\mid \text{there exists $\vec
\alpha\in U_K^+$ such that $\deg\vec\alpha=\sigma$}\}\nonumber\\
&=\min\{\sigma\in{\mathbb N}\mid \text{there exists $\tilde{\vec
\alpha}\in B$ such that $\deg\tilde{\vec\alpha}=\sigma$}\}\nonumber\\
&=\min\{\sigma\in{\mathbb N}\mid \text{there exists $x\in \*{K_{\pL}}$
such that $x\in \ker \phi_{\pL}$ and $\deg_{\pL} x=\sigma$}\}. 
\tag*{$\fin$}
\end{align*}
\end{lemma}

\section{Cyclic extensions of prime power degree}\label{S3}

We study the extended genus field $\ge K$ of a finite cyclic extension
$K/k$ of degree $l^n$ with $l$ a prime number and $n\geq 1$.
We will assume that the extension $K/k$ is geometric, that is, the field
of constants of $K$ is $\F$.
We consider four type of primes $l$:
\las
\item $l=p$, where $p$ is the characteristic of $k$, the Artin-Schreier-Witt
case,
\item $l\neq p$ and $l\nmid q-1$,
\item $l^n|q-1$, the Kummer case,
\item $l^{\rho}|q-1$ wiith $1\leq \rho<n$ and $l^n\nmid q-1$, the
``semi-Kummer'' case.
\end{list}

\subsection{The Artin-Schreir-Witt case: $l=p$}\label{SubsecEx1}

Since $\H$ is a subgroup of the inertia group of $\p$ in $E/k$
and the order of this last group is a divisor of $q-1$, the order of
$\H$ is relative prime to $p$. Hence $\H=\{\Id\}$. It follows
that $\ge{(\g E^{\H})}=\ge E$.

\subsection{Case $l\neq p$ and $l\nmid q-1$}\label{SubsecEx2}

By the same reason of Subsection \ref{SubsecEx1}, the order of
$\H$ is relative prime to $l$. Thus $\H=\{\Id\}$ and $\ge{(\g E^{\H})}
=\ge E$.

\subsection{The Kummer case: $l^n|q-1$}\label{Sc4.2*}

\subsubsection{The genus field in the cyclotomic case}\label{Ss4.2}

Let $K=F=\Ku{l^n}{\*D}$ with $D=\poly$
be a Kummer cyclic extension of $k$. Let $X=
\langle\chi\rangle$ be the group of Dirichlet characters associated to
$F$. Note that for any $\nu\in {\ma N}$ 
relatively prime to $l$, the
field associated to $\chi^{\nu}$ is $F$ since $X=\langle\chi^{\nu}\rangle$.
The above corresponds to the fact that $F=\Ku{l^{n}}{\*{(D^{\nu})}}$. 

When $D=P\in R_T^+$
we have that the character associated to $F$ is $\xbinom{}P_{l^{n}}$,
the Legendre symbol which is defined as follows: if $P$ is of degree
$d$, then for any $N\in R_T$ with $P\nmid N$, $N\bmod P\in
(R_T/\langle P\rangle)^*\cong {\ma F}_{q^d}^*$. Then 
$\xbinom NP_{l^{n}}$
is defined as the unique element of ${\ma F}_{q^d}^*$ such that
$N^{\frac{q^d-1}{l^{n}}}\equiv \xbinom NP_{l^{n}}\bmod P$. We have that 
$\xbinom {}P_{l^{n}}$ is the character associated to $\Ku {l^{n}}{\*P}$
(see \cite[Proposici\'on 9.6.1]{RzeVil2017}). Let us denote
$\chi_P=\xbinom {}P_{l^{n}}$. Then $\chi_P^{\nu}$ is the character
associated to $\Ku {l^{n}}{(P^{\nu})^*}$.

Hence, if $\chi_D$ is the character associated
to $\Ku {l^n}{\*D}$, then $\chi_{D}=
\prod_{j=1}^r\chi_{P_{j}}^{\alpha_{j}}$.

\begin{remark}\label{REx1.1+3}{\rm{
Let $l$ be a prime number different to $p$, the characteristic of
$K$, such that $l^{\kappa}|q-1$, $\kappa\geq
1$. We have that $-1\in (\f)^{l^{\kappa}}$ for all $l$ and 
all $\kappa$ except when $l=2$ and $2^{\kappa+1}\nmid q-1$.
}}
\end{remark}

\begin{proof}
If $l$ is odd, $(-1)^{l^{\kappa}}=-1$ for all $\kappa\in{\mathbb N}$. 
Let $l=2$. If $2^{\kappa+1}|q-1$, $\F$ contains a primitive $2^{
\kappa+1}$ root of unity $\xi$. Let $\mu:=\xi^{2^{\kappa}}$.
Then $\mu\neq 1$ and $\mu^2=1$. Thus $\mu=-1$.

On the other hand, if $2^{\kappa+1}\nmid q-1$, $\xi\notin \f$.
Now if we had $\mu=-1=\rho^{2^{\kappa}}$ for some $\rho
\in \f$, then $\rho$ is a primitive $2^{\kappa+1}$ root of
unity, contrary to our hypothesis.
\end{proof}

\begin{corollary}\label{CEx1.1+4}
For any prime number $l$ such that $l^{\kappa}|q-1$,
with $\kappa\in{\mathbb N}$, and
$D\in R_T$ we have $\Ku{l^{\kappa}}{D^*}=\Ku{l^{\kappa}}{D}$,
except when $\deg D$ is odd, $l=2$, and $2^{\kappa+1}\nmid q-1$.
We also have that if $\gamma\in\f$ and 
$\varepsilon=(-1)^{\deg D}\gamma$, then
$\kum{\kappa}{\varepsilon}=
\kum{\kappa}{\gamma}$ with the same exception.
\end{corollary}

\begin{proof} If $\deg D$ is even, $(-1)^{\deg D}=1$. If $\deg D$ is odd,
$(-1)^{\deg D}=-1\in (\f)^{l^{\kappa}}$ except when $l=2$
and $2^{\kappa+1}\nmid
q-1$. The same for $\kum{\kappa}{\varepsilon}$ and
$\kum{\kappa}{\gamma}$.
\end{proof}

In general, for a radical extension, we have:

\begin{theorem}\label{T3.3} Let $F=\Ku{s}{\gamma D}$ be a geometric separable
extension of $k$, $\gamma\in\f$, and let $D=\poly\in R_T$. Then 
\begin{gather*}
e_{F/k}(P_{j})=\frac s{\gcd(\alpha_j,s)},\quad 1\leq j
\leq r\quad \text{and}\\
e_{\infty}(F|k):=e_{F/k}(\p)=\frac s{\gcd(\deg D,s)}.
\end{gather*}
\end{theorem}

\begin{proof} 
See \cite[\S 5.1]{MaRzVi2017}.
\end{proof}

As a consequence we obtain the following result for a cyclic
cyclotomic Kummer extension $F=k(\sqrt[l^n]{\*D})$. Let $X$
be the group of Dirichlet characters associated to $F$ and let
$Y=\prod_{P\in R_T^+} X_P$ be the group
associated to $M$, the maximal
cyclotomic extension of $F$ unramified at the finite primes. 

Let $P=P_{j}$, $X=X_P=\langle\chi_P\rangle$ and let $F_P$ be
the field associated to $X_P$.
Then $F_P$ is cyclotomic, $P$ is the only ramified prime in
$F_P/k$ and $P$ is tamely ramified in $F_P/k$. This implies that $F_P
\subseteq \cicl P{}$ and $\Gal(\cicl P{}/k)\cong C_{q^{d_P}-1}$ with
$d_P:=\deg P$. Therefore $F_P$ is the only field of degree $o(
\chi_P)=:l^{\beta_P}$ over $k$. Since
$F_P/k$ is a Kummer extension, it follows that $F_P
=\Ku{l^{\beta_P}}{\*P}$. 

\begin{theorem}\label{T3.1}
The maximal unramified
cyclotomic extension of $F=\Ku {l^n}{\*D}$ at the finite primes is
$M:=k(\sqrt[l^n]{(P_1^{\alpha_1})^*},\ldots, \sqrt[l^n]{(P_r^{\alpha_r})^*})$.
In other words, 
\[
\ge F=(\sqrt[l^n]{(P_1^{\alpha_1})^*},\ldots, 
\sqrt[l^n]{(P_r^{\alpha_r})^*}).
\]
\end{theorem}

\begin{proof}
It follows since $M$ corresponds to the group of
Dirichlet characters $Y=\prod_{P\in R_T^+}X_P$
and the field associated to $X_P$ is $F_P
=\Ku{l^{\beta_P}}{\*P}$ for each $P\in R_T^+$. The 
result follows.
\end{proof}

\begin{remark}\label{R3.2}{\rm{
Let $\alpha=l^a b$ with $\gcd(b,l)=1$ 
and $a < n$. Then $\Ku {l^n}{(P^{\alpha})^*}=
\Ku {l^{n-a}}{\*P}$ and
\[
\ge F=k(\sqrt[l^{n-a_1}]{P_1^*},\ldots, \sqrt[l^{n-a_r}]{P_r^*})=
F_1\cdots F_r,
\]
with $F_{j}=\Ku {l^{n-a_{j}}}{P_{j}^*}$, $1\leq j\leq r$.
}}
\end{remark}

Another proof of Theorem \ref{T3.1} is using Abhyankar's Lemma. On the one
hand we have that 
\[
[M:k]=\prod_{P\in R_T^+}|X_P|=\prod_{j=1}^r
|X_{P_{j}}|= \prod_{j=1}^r e_{M/k}(P_{j})=\prod_{j=1}^r 
l^{n-a_{j}}.
\]
On the other
hand if, $F_{j}=\Ku{l^{n-a_{j}}}{\*{(P_{j})}}$, from Abyankar's Lemma, $F F_{j}/F$ is
unramified at every finite prime, so $F F_1\cdots F_r/F$ is unramified at
the finite primes and $F\subseteq F_1\cdots F_r$. Hence $F_1\cdots
F_r\subseteq \ge F$ and $[F_1\cdots F_r:k]=[M:k]$. Therefore $M=
F_1\cdots F_r$.

The genus field of a cyclotomic cyclic extension, is given by the following theorem.

\begin{theorem}\label{Pengln} Let $E =k(\sqrt[l^n]{D^*})$, 
with $D=P_1^{\alpha_1}\cdots P_r^{\alpha_r}$, $1\leq \alpha_j\leq 
l^n-1$, $\alpha_j=b_jl^{a_j}$ with $\gcd(b_j,l)=1$, $1\leq j \leq r$, 
$P_1,\dots, P_r\in R_T^+$ different 
monic irreducible polynomials with $\deg P_j=c_jl^{d_j}$, 
$\gcd(c_j,l)=1$, $1\leq j\leq r$. We order the polynomials 
$P_1,\dots,P_r$ such that $0=a_1\leq\cdots\leq a_r\leq n-1$.

Let $E_{\eu {gex}}:=E_1\cdots E_r$ with $E_j=k(\sqrt[l^{n-a_j}]{P_j^*})$, $1\leq j\leq r$. Let
\begin{align*}
    e_\infty(E|k)=l^t\;\textit{with }\;t&=n-{\min}\{n,v_l(\deg D)\},\\
    &\\
    e_\infty(E_{\eu {gex}}|k)=l^m\;\textit{with }\;m&=\underset{1\leq j\leq r}
    {\max}v_l(e_\infty(E_j|k))\\
    &={\max}\{n-a_j-{\min}\{n-a_j,d_j\} \quad 1\leq j\leq r\}.
\end{align*}

Let $i_0$, $1\leq i_0\leq r$, be such that $n-a_{i_0}-
{\min}\{n-a_{i_0},d_{i_0}\}=m$ and $n-a_j-d_j<m$ for $j>{i_0}$.
For $m>0$ we have $\gcd(\deg P_{i_0},l^n)=l^{d_{i_0}}$, and 
therefore there exist $a, b\in
{\ma Z}$ such that $a\,\deg P_{i_0}+bl^n=l^{d_{i_0}}$. 
For $j<{i_0}$, we have $d_{i_0}\leq d_j$. Let $z_j:=-ac_jl^{d_j-d_{i_0}}$. 
For $j>{i_0}$, let $y_j\equiv -c_jc_{i_0}^{-1}\mod l^{n}\in{\ma Z}$.

Then $$\g E=F_1\cdots F_r,$$ where $F_j=E_j$ with $1\leq j\leq r$ 
if $m=t$, i.e, $\g E=E_{\eu {gex}}$, and if $m>t\geq 0$, then
\begin{gather}\label{knorr}
    F_j:=
    \begin{cases}
    k\left(\sqrt[l^{n-a_j}]{P_jP_{i_0}^{z_j}}\right)&\text{if $j<{i_0}$,}\\
    k\left(\sqrt[l^{d_{i_0}+t}]{P_{i_0}^*}\right)&\text{if $j={i_0}$,}\\
        k\left(\sqrt[l^{n-a_j}]{P_jP_{i_0}^{y_jl^{d_j-d_{i_0}}}}\right)&
    \text{if $j>{i_0}$ and $d_j\geq d_{i_0}$},\\
    k\left(\sqrt[l^{n-a_j+d_{i_0}-d_j}]{P_j^{l^{d_{i_0}-d_j}}P_{i_0}^{y_j}}
    \right)&\text{if $j>{i_0}$ and $d_{i_0}> d_j$.}
    \end{cases}
\end{gather}
\end{theorem}

\begin{proof}
See \cite[Theorem 3.2]{MoReVi2019}.
\end{proof}

\begin{remark}\label{REx4.11+1}{\rm{
When $m=t$, we may also use the description of $\g K$ given in 
the case $m>t$.
}}
\end{remark}

\subsubsection{The genus field in the general case}\label{S3.2}

The genus field of a general $l^n$ cyclic extension of $k$ is given 
by the following theorem.

\begin{theorem}\label{Tprin} 
Let $K =k(\sqrt[l^n]{\gamma D})\subseteq \cicl D{}_u$, with $\gamma\in \F^*$, 
$D=P_1^{\alpha_1}\cdots P_r^{\alpha_r}$, $1\leq \alpha_j\leq l^n-1$, $\alpha_j=b_j
l^{a_j}$ with $\gcd(b_j,l)=1$, $1\leq j \leq r$, $P_1,\dots, P_r\in R_T^+$ 
different polynomials and some $u\in{\ma N}$.
We order the polynomials $P_1,\dots,P_r$ so that 
$0=a_1\leq\cdots\leq a_r\leq n-1$. Let $E=K_u\cap\cicl D{}$, $t$ as in 
Theorem {\rm{\ref{Pengln}}} and $\alpha=v_{l}(|\H|)$. Let $\H':=\H\mid_{\g E}$. Then 
$E_{\eu{ge}}^{\H'}=F_1\cdots F_{{i_0}-1}F_{{i_0}+1}\cdots F_r (\sqrt[l^{d_{i_0}+
(t-\alpha)}]{P_{i_0}^*})$ where $F_j$ are given in {\rm{(\ref{knorr})}} for all $j$. Thus
\[
 \g K=E_{\eu{ge}}^{\H'}K=\prod_{\stackrel{i=1}{i\neq i_0}}^r F_i K(\sqrt[l^{d_{i_0}
 +(t-\alpha)}]{P_{i_0}^*}).
\]

Further, if $d={\rm{min}}\{n,v_l(\deg D)\}$, we have
\[
|\H|=l^\alpha=[\F(\sqrt[l^n]{(-1)^{\deg D}\gamma})
:\F(\sqrt[l^d]{(-1)^{\deg D}\gamma})].
\]
\end{theorem}

\begin{proof}
See \cite[Theorem 4.1]{MoReVi2019}.
\end{proof}

The general structure of $\ge K$ when $K/k$ is a finite
$l$--Kummer extension for a prime number $l$, is given by
$\ge K=DK$ with $D$ a field satisfying $\ge{(\g E^{\H})}
\subseteq D\subseteq \ge E$.

With notations given above, particularly in Theorem \ref{Pengln},
we consider first the case $m>t$. 

\begin{proposition}\label{P3.5(-1)}
If $m>t$ then $i\geq 2$ and
there exists $j<i$ such that $m=n-a_{j}-d_{j}=n-a_i-d_i$.
\end{proposition}

\begin{proof}
First assume that $i\geq 2$. Suppose that for all $1\leq j\leq i-1$
we have $n-a_j-d_j<n-a_i-d_i=m$. 

We have that for all $j\neq i$, $n-a_i-d_i>n-a_j-\min\{n-a_j,d_j\}
\geq n-a_j-d_j$. 
Thus
\begin{gather}
 n-a_j-d_j<n-a_i -d_i, \quad\text{so that}
\quad a_{i}+d_i< a_{j}+d_j\quad\text{for all}
\quad j\neq i. \label{Ec4.7}
\end{gather}

We have 
\begin{align*}
\deg D&=\sum_{j=1}^r \alpha_{j}\deg P_j=
\sum_{j=1}^r b_{j}l^{a_{j}}c_jl^{d_j}\\
&=b_{i} c_il^{a_{i}+d_i}+l^{a_{i}
+d_i+1}\Big(\sum_{\substack{j=1\\ j\neq i}}^r b_{j}c_j
l^{a_{j}+d_j-a_{i}-d_i-1}\Big).
\end{align*}
Hence $v_l(\deg D)=l^{a_{i}+d_i}$.
It follows that
\begin{gather*}
n-\min\{n,v_l(\deg D)\}\geq n-v_l(\deg D)
=n-a_{i}-d_i=m\quad\text{and}\\
t \geq n-\min\{n,v_l(\deg D)\} \geq m\geq t.
\end{gather*}
Therefore $m=t$ contrary to our assumption.
Thus, there exists $1\leq j\leq i-1$ with $n-a_j+d_j=
n-a_i-d_i=m$.

The same argument shows that if $i=1$, then $m=t$.
\end{proof}

From Theorems \ref{Pengln} and \ref{Tprin} 
we have that for all $j\neq i$, we have
$e_{P_j}(\g{E^{\H}}|k)=e_{P_j}(K|k)=e_{P_j}(E|k)$.
In the case when there exists $1\leq j\leq i-1$ such that
$n-a_j-d_j=n-a_i-d_i$ we obtain that
\begin{gather*}
\begin{align*}
e_{P_i}(E_j|k)&=e_{P_i}\big(\Ku{l^{n-a_j}}{P_jP_i^{-ac_jl^{d_j-d_i}}}|k\big)=
l^{n-a_j-d_j+d_i}\\
&=l^{n-a_i}=e_{P_i}(E|k)=e_{P_i}(K|k).
\end{align*}
\intertext{Hence}
e_{P_i}(\g{E^{\H}}|k)=e_{P_i}(K|k)=e_{P_i}(E|k).
\end{gather*}

From the above, the following result is immediate.
 
\begin{proposition}\label{P3.6(2)}
If there exists $1\leq j\leq i-1$ such that $n-a_j-d_j=
n-a_i-d_i$, in particular when $m>t$,
then $\ge{(\g{E^{\H}})}=\ge E$.
$\fin$
\end{proposition}

The first main result on extended genus fields, is the following:

\begin{theorem}\label{TEx1.1}
With the above notations we have that $\ge K=\ge E K$, except in the
following case:
\l
\item $K\neq E$,
\item $\H\neq\{\Id\}$,
\item $t=m>0$,
\item $m=n-a_i-\min\{n-a_i,d_i\}>n-a_j-\min\{n-a_j,d_j\}$
for all $j\neq i$,
\end{list}
\end{theorem}

\begin{proof}
If $\H=\{\Id\}$ the result follows from \cite[Theorem 4.5]{RzeVil2023}.
If $E=K$, $K$ is cyclotomic and therefore $\H=\{\Id\}$. If $m>t$, then
from Proposition \ref{P3.5(-1)}, we have that $e_{P_j}(\g E^{\H}|k)=e_{P_j}
(\g E|k)$ for all $1\leq j\leq r$
and therefore $\ge{(\g E^{\H})}=\ge E$. The result follows from
\cite[Theorem 4.5]{RzeVil2023}. If
$n-a_i-\min\{n-a_i,d_i\}=n-a_j-\min\{n-a_j,d_j\}$ for some
$j\neq i$, then $e_{P_j}(\g E^{\H}|k)=e_{P_j}
(\g E|k)$ for all $1\leq j\leq r$
and therefore $\ge{(\g E^{\H})}=\ge E$.
If $t=0$, we have that $\p$ in unramified in $E/k$ so that $\H=\{\Id\}$.
\end{proof}

\subsubsection{The especial case}\label{SEx6}

We now consider the especial case, that is, the exception
given in Theorem \ref{TEx1.1}. 

Let $K=k\big(\sqrt[l^{n}]{\gamma D},\big)$
be a geometric separable
extension of $k$, with $\gamma\in\f$ and let $D=\poly\in R_T$,
with $P_1,\ldots,P_r\in R_T^+$ distinct, $1\leq \alpha_j
\leq l^n-1$, $1\leq j\leq r$. Let $\alpha_j=l^{a_j}b_j$, $l\nmid
b_j$, $\deg P_j=c_jl^{d_j}$, $l\nmid c_j$. 
We assume that we have
the exception given in Theorem \ref{TEx1.1}. Let $E=
k\big(\sqrt[l^n]{D^*}\big)$, and $\deg D=l^{\delta}c$ with $l\nmid c$.
Then $e_{\infty}(K|k)=l^{n-
\delta}=l^t=l^m=l^{n-a_i-d_i}$, so that $\delta = a_i+d_i$.

Since $m=t>0$, we have $m=n-a_i-\min\{n-a_i,d_i\}=n-a_i-d_i$ and
$n-a_i-d_i > n-a_j-d_j$ for all $j\neq i$. We also have that $\varepsilon
:= (-1)^{\deg D}\gamma\notin (\f)^l$.

\begin{lemma}\label{LEx1.1+1}
We have $\g E=\g E^{\H} E$. It also holds that
$EK/E$ and $EK/K$ are extensions of
constants and $EK=E\kum{n}{\varepsilon}=K\kum{n}{\varepsilon}$.
That is, 
\[
EK=E\big(\rad{n}{\varepsilon}\big)=K\big(\rad{n}{\varepsilon}\big).
\]

We also have that $\g EK/\g K$ and $\g EK/\g E$ are 
extensions of constants. Furthermore, $\g EK=\g K\kum{n}{\varepsilon}=
\g K\big(\rad{n}{\varepsilon}\big)$ and $\g E K=
\g E\kum{n}{\varepsilon}=\g E\big(\rad{n}{\varepsilon}\big)$.
\end{lemma}

\begin{proof}
The extension $\g E/\g E^{\H}$ is fully ramified at the infinite prime $\p$.
Since $\g E^{\H}\subseteq \g E^{\H} E\subseteq \g E$ and since $e_{\infty}
(E|k)=e_{\infty}(\g E|k)$, it follows that $\g E=\g E^{\H} E$.

Now $EK=\Ku{n}{\gamma D}\Ku{n}{(1)^{\deg D}D} = E\big(\rad{n}{
\varepsilon}\big)=K\big(\rad{n}{\varepsilon}\big)$.

We also have 
$\g EK=\g E^{\H} EK=\g E^{\H}K EK=\g KK\big(\rad{n}{\varepsilon}
\big)=\g K\big(\rad{n}{\varepsilon}\big)$. Therefore
\begin{gather*}
\g EK=
\g EEK=\g E\big(\rad{n}{\varepsilon}\big).
\end{gather*}
\end{proof}

\begin{corollary}\label{CEx1.1+2}
The field of constants of $\g EK$ is $\kum{n}{\varepsilon}$. $\fin$
\end{corollary}

\begin{theorem}\label{TEx1.1+4}
In the exceptional case given in Theorem {\rm{\ref{TEx1.1}}}, we have
that $\g E=\ge E$, the field of constants of $\g K$ is ${\mathbb F}_{
q^{\deg_K\p}}$, the field of constants of $\ge E K$ is $\kum{n}{
\varepsilon}$.
\end{theorem}

\begin{proof}
Since $m=t$, we have $\g E=\ge E$.
\end{proof}

Later on, we will see that the field of constants of $\g K$ is
$\kum{\delta}{\gamma}$.

We fix an infinite prime $\pL$ of $K$ and we denote $K_{\infty}
:=K_{\pL}$.

Since $K_{\infty}=k_{\infty}\big(\sqrt[l^n]{\gamma D}\big)$, $\deg D
=d=l^{\delta}c$
with $l\nmid c$, we have 
\begin{align*}
D(T)&=T^{d}+a_{d-1}T^{d-1}+\cdots+a_1T+a_0\\
&=T^d\Big(1+a_{d-1}\big(\frac 1T\big)+\cdots a_1
\big( \frac 1T\big)^{d-1}+a_0\big(\frac 1T\big)^d\Big)=
T^d D_1(1/T).
\end{align*}
We have that $D_1(1/T)\in U_{\infty}^{(1)}$ and since
$l$ is different to the characteristic, it follows that $(U_{\infty}^{
(1)})^{l^n}=U_{\infty}^{(1)}$. Therefore
\[
K_{\infty}=k_{\infty}\big(\sqrt[l^n]{\gamma D}\big)=
k_{\infty}\big(\sqrt[l^n]{\gamma T^dD_1(1/T)}\big)=
k_{\infty}\big(\sqrt[l^n]{\gamma T^{l^{\delta}c}}\big).
\]

Since $\gcd(l,c)=1$, there exists $c_1\in{\mathbb Z}$ such
that $cc_1\equiv -1\bmod l^n$. Thus
\[
K_{\infty}=
k_{\infty}\big(\sqrt[l^n]{\gamma^{c_1}T^{l^{\delta cc_1}}}\big)=
k_{\infty}\big(\sqrt[l^n]{\gamma^{c_1}(1/T)^{l^{\delta}}}\big)
=k_{\infty}\big(\sqrt[l^n]{\gamma^{c_1}\pi_{\infty}^{l^{\delta}}}\big).
\]

Now $[K_{\infty}:k_{\infty}]=e_{\infty}(K|k)f_{\infty}(K|k)=
l^{n-\delta}\deg_K \pL$. Set $K_0=k\big(\sqrt[l^{\delta}]{\gamma D}\big)
\subseteq K$. We have that $e_{\infty}(K_0|k)=\frac{l^{\delta}}
{\gcd(\deg D,l^{\delta})}=\frac{l^{\delta}}{\mcd({l^{\delta}c,l^{\delta})}}
=\frac{l^{\delta}}{l^{\delta}}=1$ and $e_{\infty}(K|K_0)=e_{\infty}(
K|k)=l^{n-\delta}=[K:K_0]$. Therefore $\p$ is fully ramified in
$K/K_0$.

\[
\xymatrix{
K=k\big(\sqrt[l^n]{\gamma D}\big)\ar@{-}[d]_{l^{n-\delta}}\\
K_0=k\big(\sqrt[l^{\delta}]{\gamma D}\big)\ar@{-}[d]\\ k
}
\qquad \qquad
\xymatrix{
K_{\infty}\ar@{-}[d]^{e_{\infty}(K|k)=
l^{n-\delta}}\\ K_{0,\infty}\ar@{-}[d]^{
f_{\infty}(K|k)=\deg_K\pL=l^{\lambda}}\\k_{\infty}
}
\]

We have $f_{\infty}(K|k)=f_{\infty}(K_0|k)=f_{\infty}(K_{0,\infty}
|k_{\infty})=\deg_K\pL$ and
\[
K_{0,\infty}=k_{\infty}\big(\sqrt[l^{\delta}]{\gamma D}\big)=
k_{\infty}\big(\sqrt[l^{\delta}]{\gamma T^d}\big)=
k_{\infty}\big(\sqrt[l^{\delta}]{\gamma T^{l^{\delta}c}}\big)=
k_{\infty}\big(\sqrt[l^{\delta}]{\gamma}\big).
\]

\begin{lemma}\label{LEx1.6+2}
The field of constants of $\g K$ is $\kum{\delta}{\gamma}$. $\fin$
\end{lemma}

Let
\begin{gather*}
f_{\infty}(K|k)=\deg_K\pL=[K_{0,\infty}:k_{\infty}]=
[\F\big(\sqrt[l^{\delta}]{\gamma}\big):\F]=:l^{\lambda}.
\intertext{We also have}
EK=E\big(\sqrt[l^n]{\varepsilon}\big)=
K\big(\sqrt[l^n]{\varepsilon}\big).
\intertext{Then, $EK/E$ is an extension
of constants and, since $\deg_E\p=1$, it follows that}
f_{\infty}(EK|k)=f_{\infty}(EK|E)=[EK:E]=[\F\big(\sqrt[l^n]{
\varepsilon}\big):\F]=:l^{\nu}.
\end{gather*}

Now, $|\H|=f_{\infty}(EK|K)=\frac{f_{\infty}(EK|k)}{f_{\infty}(K|k)}=
\frac{l^{\nu}}{l^{\lambda}}=l^{\nu-\lambda}=:l^u$.
The field of constants of $\g K$ is ${\mathbb F}_{q^{l^{\lambda}}}$
and the field of constants of $\g E K$ is $\F\big(\sqrt[l^n]{
\varepsilon}\big)={\mathbb F}_{q^{f_{\infty}(EK|k)}}={\mathbb F}_{
q^{l^{\nu}}}$. We have ${\mathbb F}_q\big(\sqrt[l^{\delta}]{\gamma}
\big)={\mathbb F}_q\big(\sqrt[l^{\delta}]{\gamma^{c_1}}\big)=
{\mathbb F}_{q^{l^{\lambda}}}$. Therefore $[\F\big(\sqrt[l^{\delta}]{
\gamma^{c_1}}\big):\F]=l^{\lambda}$. From 
Theorem \ref{TEx1.6}, we obtain that
$\gamma^{c_1}\in (\f)^{l^{\delta-\lambda}}\setminus
(\f)^{l^{\delta-\lambda+1}}$.

Let $\gamma^{c_1}=\theta^{l^{\delta-\lambda}}$,
with $\theta\in\f$ and $\theta\notin (\f)^l$. Then
\[
K_{\infty}=k_{\infty}\big(\sqrt[l^n]{\gamma^{c_1}\pi_{\infty}^{
l^{\delta}}}\big)=k_{\infty}\big(\sqrt[l^n]{\theta^{l^{\delta-
\lambda}}\pi_{\infty}^{l^{\delta}}}\big)=k_{\infty}\big(\sqrt[l^{
n-\delta+\lambda}]{\theta\pi_{\infty}^{l^{\lambda}}}\big).
\]

The element $\xi:=\sqrt[l^{n-\delta+\lambda}]{\theta\pi_{\infty}^{
l^{\lambda}}}$ satisfies $\xi^{l^{n-\delta+\lambda}}=\theta
\pi_{\infty}^{l^{\lambda}}$, that is, $\xi$ is a root of $X^{l^{n-
\delta+\lambda}}-\theta\pi_{\infty}^{l^{\lambda}}\in
k_{\infty}[X]$. Since $[K_{\infty}:k_{\infty}]=l^{n-\delta+\lambda}$,
the polynomial 
\[
X^{l^{n-\delta+\lambda}}-\theta\pi_{\infty}^{l^{\lambda}}
\]
is irreducible. We also have $K_{0,\infty}=k_{\infty}\big(\sqrt[l^{
\delta}]{\gamma}\big)=k_{\infty}{\mathbb F}_{q^{l^{\lambda}}}$
and $K/K_0$ is fully ramified at $\p$.

Set $\tilde\Pi:=\xi=\sqrt[l^{n-\delta+\lambda}]{\theta
\pi_{\infty}^{l^{\lambda}}}$. Then $\tilde \Pi^{l^{n-\delta+\lambda}}=
\theta\pi_{\infty}^{l^{\lambda}}$ and
\begin{align*}
v_{\pL_{\infty}}\big(\tilde\Pi^{l^{n-\delta+\lambda}}\big)&=l^{n-\delta+
\lambda}v_{\pL_{\infty}}(\tilde\Pi)=e_{\infty}(K_{\infty}|k_{
\infty})v_{\infty}\big(\theta\pi_{\infty}^{l^{\lambda}}\big)\\
&=e_{\infty}(K|k)l^{\lambda}=l^{n-\delta} l^{\lambda}=l^{n-\delta
+\lambda}.
\end{align*}
Hence $v_{\pL_{\infty}}(\tilde\Pi)=1$ and $\tilde\Pi$ is prime 
element of $K_{\infty}$. We also have 
\[
\deg_{K_{\infty}}
\tilde\Pi=\deg_{K_{\infty}}\pL_{\infty}v_{\pL_{\infty}}(\tilde\Pi)=
\deg_K\p\cdot 1=l^{\lambda}.
\]

Now $\theta\notin (\f)^l$. Let $\zeta_{l^{
n-\delta+\lambda}}$ be a primitive $l^{n-\delta+\lambda}$-th
root of unity and set $\N_{\infty}:=\N_{K_{\infty}|k_{\infty}}$.
We have
\begin{align*}
\Irr(\tilde\Pi, X,k_{\infty})&=X^{l^{n-\delta+\lambda}}-\theta
\pi_{\infty}^{l^{\lambda}}=\prod_{j=0}^{l^{n-\delta+\lambda}-1}
\big(X-\zeta_{l^{n-\delta+\lambda}}^{j}\tilde \Pi\big).
\intertext{Thus}
\N_{\infty}\tilde \Pi&=\prod_{j=0}^{l^{n-\delta+\lambda}-1}\big(
\zeta_{l^{n-\delta+\lambda}}^{j}\tilde \Pi\big)=(-1)^{l^{n-\delta
+\lambda}} \prod_{j=0}^{l^{n-\delta+\lambda}-1}\big(-\zeta_{
l^{n-\delta+\lambda}}^{j}\tilde \Pi\big)\\
&=(-1)^{l^{n-\delta+\lambda}}
\big(-\theta\pi_{\infty}^{l^{\lambda}}\big)=(-1)^{l^{n-\delta+
\lambda}+1}\theta\pi_{\infty}^{l^{\lambda}}.
\end{align*}

Now we consider a generic element of $Y\in\*{K_{\infty}}$:
\begin{gather*}
Y=\tilde\Pi^{s}\Lambda w, \quad\text{with}\quad
s\in{\mathbb Z},\quad
\Lambda\in {\mathbb F}_{q^{l^{\lambda}}},\quad\text{and}
\quad w\in U_{K_{\infty}}^{(1)}.
\intertext{Then}
\begin{align*}
\N_{\infty}\tilde \Pi^{s}&=\big(\N_{\infty}\tilde\Pi\big)^{s}=
(-1)^{(l^{n-\delta+\lambda}+1)s}\theta^{s}\pi_{\infty}^{
l^{\lambda}s},\\
\N_{\infty}\Lambda&=\N_{K_{0,\infty}|k_{\infty}} 
\big(\N_{K_{\infty}|K_{0,\infty}}\Lambda\big)=
\N_{K_{0,\infty}|k_{\infty}}(\Lambda^{l^{l^{n-\delta}}})
=\big(\N_{K_{0,\infty}|k_{\infty}}
\Lambda\big)^{l^{n-\delta}},\\
\N_{\infty} w&= v\in U_{\infty}^{(1)}.
\end{align*}
\end{gather*}

It follows that 
\begin{align*}
\phi_{\pL_{\infty}}(Y)&=\phi_{\infty}(\N_{\infty}(Y))=
\phi_{\infty}\big((-1)^{(l^{n-\delta+\lambda}+1)s}\theta^{
s}\pi_{\infty}^{l^{\lambda}s}
(\N_{K_{0,\infty}|k_{\infty}}\Lambda)^{l^{n-\delta}}
v\big)\\
&=(-1)^{(l^{n-\delta+\lambda}+1)s}\theta^{s}\big(
\N_{K_{0,\infty}|k_{\infty}}\Lambda\big)^{l^{n-\delta}}\\
&=(-\theta)^{s}\big[(-1)^{l^{\lambda}}
\big(\N_{K_{0,\infty}|
k_{\infty}}\Lambda\big)\big]^{l^{n-\delta}}.
\end{align*}

Therefore $Y\in\ker\phi_{\pL_{\infty}}\iff \text{there exists $\Lambda
\in {\mathbb F}_{q^{l^{\lambda}}}$}$ such that
\[
(-\theta)^{s}\big[(-1)^{l^{\lambda}}
\big(\N_{K_{0,\infty}|
k_{\infty}}\Lambda\big)\big]^{l^{n-\delta}}=1.
\]
Now, $\N_{K_{0,\infty}|k_{\infty}}{\mathbb F}_{q^{l^{\lambda}}}
=\F$, thus $\N_{\infty}\*{{\mathbb
F}_{q^{l^{\lambda}}}}=\big(\f\big)^{l^{n-\delta}}$ and
$\Lambda\in {\mathbb F}_{q^{l^\lambda}}$. 
Hence $-\theta^{ s}\in (\f)^{l^{n-\delta}}$.

\subsubsection{Case $n=1$} This case was considered in
\cite{HeVi2023}.

\subsubsection{Case $n\geq 2$}. We now assume that $n>1$.
We always have that, since $\theta\notin (\f)^l$, then
$- \theta\notin (\f)^l$ because $n\geq 2$
and therefore $-1\in(\f)^l$ (see
Remark \ref{REx1.1+3}). Hence $-\theta^{ s}\in
(\f)^{l^{n-\delta}} \iff l^{n-\delta}|s$. That is $\ker \phi_{
\pL_{\infty}}=\{Y=\tilde\Pi^s\Lambda w\mid l^{n-\delta}|s\}$.

Because $\deg Y=\deg\big(\tilde{\Pi}^s\Lambda w\big)=
\deg\tilde{\Pi}\cdot v_{{\mathfrak P}_{\infty}}(Y)=l^{
\lambda}\cdot s$, it follows that
\[
\min\{\kappa\in{\mathbb N}
\mid\text{there exists $\tilde{\vec\alpha}\in B$ and $\deg\tilde{\vec
\alpha}=\kappa$}\}=l^{n-\delta +\lambda},
\]
and that the field of constants of $K_{H^+}$ is ${\mathbb F}_{q^{l^{
n-\delta+\lambda}}}$.

We have that $l^{\lambda}=[\kum {\delta}{\gamma}:\F]$, so that,
from Theorem \ref{TEx1.6}, we obtain
\begin{gather}\label{EcEx1}
\gamma\in (\f)^{l^{\delta-\lambda}}
\setminus (\f)^{l^{\delta-\lambda+1}}.
\end{gather}

On the other hand, the field of constants of $\g KE$ is
$\kum {n}{\varepsilon}$ and $[\kum {n}{\varepsilon}:\F]=
l^{\nu}$. Again, from Theorem \ref{TEx1.6}, we obtain that
\[
\varepsilon\in (\f)^{l^{n-\nu}}\setminus (\f)^{l^{n-\nu+1}}.
\]

We have $u=\nu-\lambda\geq 1$ so that $\nu\geq \lambda+1$
and $n-\nu\leq n-1$. Therefore $(-1)^{\deg D}=\pm 1\in
(\f)^{l^{n-\nu}}$. Hence $\gamma=(-1)^{\deg D}\varepsilon
\in (\f)^{l^{n-\nu}}$. It follows from (\ref{EcEx1})
that $n-\nu\leq \delta-\lambda$
and $n-\delta\leq \nu-\lambda$.

Thus, $e_{\infty}(K|k)=l^{n-\delta}|l^{\nu-\lambda}=l^u=|\H|$.
Since $|\H||e_{\infty}(K|k)=e_{\infty}(E|k)$, $l^u|l^{n-\delta}$.
Therefore, $\nu-\lambda\leq n-\delta$, so that $\nu-\lambda
=n-\delta$ and $\nu=n-\delta+\lambda$.  In particular, $|\H|=
l^u=l^{n-\delta}=e_{\infty}(K|k)$.

It follows that the
field of constants of $\g EK=\ge E K$ is $\kum{n}{\varepsilon}
={\ma F}_{q^{l^{\nu}}}={\ma F}_{q^{l^{n-\delta+\lambda}}}$.
In short, the field of constants of both,
$K_{H^+}$ and $\g EK$,
is ${\ma F}_{q^{l^{n-\delta+\lambda}}}=
{\ma F}_{q^{l^{u+\lambda}}}$. 

Thus $\g EK\subseteq K_{H^+}$ and
$\g EK\subseteq \ge K\subseteq \ge EK=\g EK$.
Therefore $\ge EK
\subseteq K_{H^+}$. Since $\ge K\subseteq \ge E K$,
we finally obtain that $\ge K=\ge E K$.

\begin{theorem}\label{TEx1.7}
In the especial case of Theorem {\rm{\ref{TEx1.1}}}, we
obtain $\ge K=\ge E K$. $\fin$
\end{theorem}

\begin{corollary}\label{CEx1.8}
For any cyclic Kummer extension $\Ku{l^n}{\gamma D}$ of $k$,
we have $\ge K=\ge E K$. $\fin$
\end{corollary}

\subsection{Semi-Kummer case: $l^{\rho}|q-1$, $\rho\geq 1$
and $l^n\nmid q-1$}\label{SubsectionExtended4}

Recall that we only need to consider the case
$\ge{(\g E^{\H})}\neq \ge E$, or, equivalently, there exists
a finite prime $P_j$ ramified in $\ge E/\g E^{\H}$, and therefore $\ge E=\g E$ (see Proposition \ref{PExtended5.4.2} ). The
prime $\p$ is fully ramified in $\ge E/\g E^{\H}$.

\begin{lemma}\label{LExtended5.19}
We have that $f_{\infty}(K|k)=\deg_K\p$.
\end{lemma}

\begin{proof}
Since the extension $K/k$ is geometric, and $\deg_k\p=1$,
we have $f_{\infty}(K|k)=f_{\infty}(K|k)\deg_k\p=\deg_K\p$.
\end{proof}

We use the following notation. Let $|\H|:=l^u=f_{\infty}(EK|K)$.
Since $\H$ is a quotient of $I$,
it follows that $l^u|l^{\rho}$ since $|I|=e_{\infty}(
\ge E|k)|q^{\deg_K\p}-1=q-1$. In particular $u\leq \rho$.

Set $l^{\lambda}:=\deg_K\p=f_{\infty}(K|k)$. We have
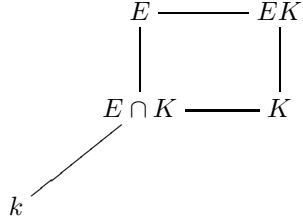
\[
\xymatrix{
&E\ar@{-}[r]\ar@{-}[d]&EK\ar@{-}[d]\\
&E\cap K\ar@{-}[dl]\ar@{-}[r]&K\\ k
}
\qquad\qquad
\begin{minipage}{5.5cm}
\begin{align*}
f_{\infty}(EK|E)&=f_{\infty}(EK|k)\\
&=f_{\infty}(EK|K)f_{\infty}(K|k)\\
&=|\H|\cdot \deg_K\p=l^{u+\lambda}.
\end{align*}
\end{minipage}
\]
Since $e_P(K|k)=e_P(E|k)$ for all $P\in R_T^+\cup\{\infty\}$,
from Abhyankar's Lemma we obtain that $e_P(EK|E)=1$
for all $P\in R_T^+\cup\{\infty\}$, that is, $EK/E$ is an unramified
extension. We will show that it is an extension of constants.

It is easy to see that $[K:k]=[E:k]$. We have, on the one hand
\begin{gather*}
[K:E\cap K]=\frac{[K:k]}{[E\cap K:k]}=\frac{[E:k]}{[E\cap K:k]}=[E:E\cap K].
\intertext{On the other hand}
[EK:E]=[K:E\cap K]\quad\text{and}\quad [EK:K]=[E:E\cap K].
\intertext{Therefore}
[EK:K]=[EK:E]=[E:E\cap K]=[K:E\cap K].
\end{gather*}

Because $EK/E$ and $EK/K$ are unramified extensions,
\[
e_P(E|E\cap K)=e_P(K|E\cap K)\quad\text{for all}\quad
P\in R_T^+\cup\{\infty\}.
\]

We have that $\g EK/\g K=\g E^{\H} K$ is an extension of constants of
degree $|\H|=l^u=[\g EK:\g K]$ (\cite[Theorem 2.2]{BaMoReRzVi2018}). 
We also have that the field of constants
of $\g K$ is ${\ma F}_{q^{\deg_K\p}}$. Hence, the field of constants of
$\g E K=\ge E K$ is ${\ma F}_{q^{\psi}}$ where $\psi=\deg_K\p \cdot
|\H|=l^{\lambda + u}=f_{\infty}(EK|k)$.

Let see that the field of constants of $EK$ is also 
${\ma F}_{q^{\psi}}={\ma F}_{q^{l^{\lambda+u}}}$, the same
of $\g EK$.

Since $K/k$ is tamely ramified, the conductor of constants 
(\cite[Theorem 3.1]{BaMoReRzVi2018})
is the minimum $\eta$ such that $K\subseteq \cicl N{}_{\eta}$.
In the notation of Theorem 3.1 and
Remark 3.2 of \cite{BaMoReRzVi2018}, we have that $\eta=
td$ where $t=f_{\infty}(K|k)=f_{\infty}(K|J)=\deg_K \p=l^{\lambda}$,
$d=f_{\infty}(EK|K)=f_{\infty}(\g EK|\g K)=|\H|=l^u$, 
and $J=K\cap{\:_n\cicl
N{}}=K\cap \cicl N{}=K\cap E$. Therefore
\begin{gather*}
\eta=l^{\lambda}\cdot l^u=l^{\lambda+u}=\psi.
\intertext{Furthermore, in the same Theorem 3.1 of 
\cite{BaMoReRzVi2018}, we have}
\eta=[K:J]=[K:K\cap E] (=[E:K\cap E]=[EK:E]=[EK:K]).
\end{gather*}

We have, see \cite{MaRzVi2013}
\[
\xymatrix{
E\ar@{-}[rr]\ar@{-}[dd]&&EK\ar@{-}[r]\ar@{--}[ddd]&(EK)_{
\eta}=E_{\eta}=K_{\eta}\ar@{-}[dd]\\
&K\ar@{-}[dl]\ar@{-}[ur]\\
E\cap K\ar@{-}[d]\ar@{-}[rrr]&&&(E\cap K)_{\eta}\ar@{-}[d]\\
k=k_{\eta}\cap E\ar@{-}[rr]&&k_{\sigma}\ar@{-}[r]&k_{\eta}
}
\]
If $EK\cap k_{\eta}=k_{\sigma}\subseteq k_{\eta}$,  then $K\subseteq
EK=k_{\sigma}E\subseteq k_{\sigma}\cicl N{}
=\cicl N{}_{\sigma}$. Since $\eta$ is the 
minimum, it follows that $\eta=\sigma$, and $EK=(EK)_{\eta}=
E_{\eta}=K_{\eta}$. Therefore, the field of constants of $EK$ is
${\ma F}_{q^{\psi}}={\ma F}_{q^{l^{\lambda +u}}}$.

\begin{proposition}\label{PExtended5.20}
The field of constants of either $\g EK$ or $EK$ is ${\ma F}_{
q^{l^{\lambda+u}}}$, where $l^{\lambda}=\deg_K\p$ and
$l^u=|\H|$. $\fin$
\end{proposition}

Furthermore, $[E_{\eta}:E]=\eta=\psi=l^{\lambda+u}=[K:J]=[K:E
\cap K]=[E:E\cap K]=[EK:K]=[EK:E]$.
\[
\xymatrix{
E\ar@/^1pc/@{-}[r]^{\eta}
\ar@{-}[d]_{\eta}&\phantom{
E_{\eta}=K_{\eta}=}
EK=E_{\eta}=K_{\eta}
\ar@{-}[d]^{\eta}\\
E\cap K\ar@{-}[r]_{\eta}&K
}
\qquad\qquad
\xymatrix{
EK\ar@{-}[d]_{l^u}^{\H}\ar@/^2pc/@{-}[dd]^{\eta={l^{\lambda+u}}}\\
E^{\H}K\ar@{-}[d]^{l^{\lambda}}\\ K
}
\]
We have that $EK/E$ is an extension of constants of degree $\eta$,
and the same is true is for the extension $EK/K$. The field of constants
of $E^{\H}K$ is ${\ma F}_{q^{l^{\lambda}}}$, that is, the same of the
field $\g K=\g E^{\H} K$.

Let ${\pL}_{\infty}$ be a prime above $\p$ and denote $K_{\infty}:=
K_{\pL_{\infty}}$, $k_{\infty}:=k_{\p}$. Let $e_{\p}(K|k)=l^{\tau}|q-1$,
that is, $l^{\tau}|l^{\rho}$ and $\tau\leq \rho$. We have
\[
[K_{\infty}:k_{\infty}]=e_{\infty}(K|k)f_{\infty}(K|k)=l^{\tau}\cdot
l^{\lambda}=l^{\tau+\lambda}.
\]

Let $k_{\infty}\subseteq F\subseteq K_{\infty}$ 
be the inertia field of $K_{\infty}/k_{\infty}$,
that is, $F:=K_{\infty}^{I_{\infty}(K_{\infty}/k_{\infty})}$ and
$[F:k_{\infty}]=l^{\lambda}$. We have  that $F/k_{\infty}$ is
unramified.

For each local field, there exists a unique unramified extension
of each degree (see \cite[Teorema 17.3.37]{RzeVil2017}).
Therefore $F=k_{\infty}{\ma F}_{q^{l^{\lambda}}}$, that is, $F/
k_{\infty}$ is an ``extension of constants'' of degree $l^{\lambda}$.
More precisely,
\[
{\mc O}_{\pL_{\infty}}/\pL_{\infty}\cong {\ma F}_{
q^{l^{\lambda}}}\quad {\text{and}}\quad  
\*F=\langle\pi_F\rangle\times \*{{\ma F}_{
q^{l^{\lambda}}}}\times U_F^{(1)},
\]
where $\pi_F=\pi_{\infty}=1/T$ is
a prime element of $\*F$.

\[
\xymatrix{
K_{\infty}\ar@{-}[d]\\F=k_{\infty} {\mathbb F}_{q^{l^{\lambda}}}
\ar@{-}[d]\\ k_{\infty}
}
\]

Let $\*K_{\infty}=\langle\tilde\Pi\rangle\times \*{\ma F}_{q^{l^{\lambda}}}
\times U^{(1)}_{K_{\infty}}$ where $\tilde\Pi$ is a prime element of 
$\*K_{\infty}$. Now, (see \cite[Teorema 17.3.14]{RzeVil2017}) we have
\[
1=v_{K_{\infty}}(\tilde\Pi)=\frac{e_{\infty}(K|k)}{[K_{\infty}:k_{\infty}]}
v_{k_{\infty}}(\N_{K_{\infty}/k_{\infty}}(\tilde\Pi))=\frac {1}{f_{\infty}(K|k)}v_{k_{\infty}}
(\N_{K_{\infty}/k_{\infty}}(\tilde\Pi)).
\]
Hence $v_{k_{\infty}}(\N_{K_{\infty}/k_{\infty}}(\tilde\Pi))=l^{\lambda}$.

We have
\[
\xymatrix{
K_{\infty}\ar@{-}[d]^{e_{\infty}=l^{\tau}} \\F\ar@{-}[d]^{f_{\infty}=
l^{\lambda}}\\ k_{\infty}
}
\qquad\qquad
\begin{minipage}{8cm}
$e_{\infty}=l^{\tau}|l^{\rho}$ and $l^{\rho}|q-1$. Therefore
the $l^{\tau}$-th primitive the unity $\zeta_{l^{\tau}}$ belongs to $F$.
Since $\H\subseteq I_{\infty}(\ge E/k)$, we also have $u\leq \tau$.
\end{minipage}
\]

It follows that $K_{\infty}/F$ is a Kummer extension, say $K_{\infty}=
F(\sqrt[l^{\tau}]{Y})$ for some $Y\in \*F=\langle\pi_{\infty}\rangle\times
\*{\ma F}_{q^{l^{\lambda}}}\times U_F^{(1)}$.

Let $Y=\pi_{\infty}^s\Lambda w$, with $s\in{\ma Z}$, $\Lambda\in
\*{\ma F}_{q^{l^{\lambda}}}$, and $w\in U_F^{(1)}$. Since $\gcd(l,p)=1$,
we have $U_F^{(1)}=(U_F^{(1)})^{l^{\tau}}$. We write $s=\alpha l^{\tau}+r$,
with $0\leq r< l^{\tau}$. Then, if $w_0^{l^{\tau}}=w$, we have
\[
K_{\infty}=F\big(\sqrt[l^{\tau}]{\pi_{\infty}^{\alpha
l^{\tau}+r}\Lambda w_0^{l^{\tau}}}
\big)=F\big(\sqrt[l^{\tau}]{\pi_{\infty}^r\Lambda}\big).
\]
Let $r=l^br_0$, with $0\leq b<\tau$ and $\gcd(l,r_0)=1$.
Set $F_1:=F\big(\sqrt[l^b]{\pi_{\infty}^{l^br_0}}\Lambda\big)=
F\big(\sqrt[l^b]{\Lambda}\big)$. Thus $F_1/F$ is unramified,
$F\subseteq F_1\subseteq K_{\infty}$ and $K_{\infty}/F$ is totally
ramified. It follows that $F_1=F$ and that $b=0$, that is, $\gcd(r,l)=1$.

Therefore $K_{\infty}=F\big(\sqrt[l^{\tau}]{\pi_{\infty}\theta}\big)$
for some $\theta\in \*{\ma F}_{q^{l^{\lambda}}}$. Set $\phi:=
\sqrt[l^{\tau}]{\pi_{\infty}\theta}$. Then $\phi^{l^{\tau}}=\pi_{\infty}
\theta$. Hence
\[
l^{\tau}v_{K_{\infty}}(\phi)=v_{K_{\infty}}(\phi^{l^{\tau}})=v_{K_{\infty}}(
\pi_{\infty}\theta)=v_{K_{\infty}}(\pi_{\infty})=e(K_{\infty}|F)v_{\infty}(
\pi_{\infty})=l^{\tau}\cdot 1.
\]
It follows that $v_{K_{\infty}}(\phi)=1$. Therefore we may take $\phi=
\tilde\Pi$ as a prime element of $F$.

Now we consider $E_{\infty}=
k_{\infty}\big(\sqrt[l^{\tau}]{\pi_{\infty}\mu}\big)$ for some
$\mu\in \*\F$. We have
\[
\xymatrix{
E_{\infty}\ar@{-}[rr]^{\eta=l^{\lambda+u}\phantom{xxxxx}}
\ar@{-}[d]_{l^u}\ar@/_3pc/@{-}[dd]_{l^{\tau}}
&&(EK)_{\infty}=E_{\infty}K_{\infty}
\ar@{-}[d]^{l^u}\\E_{\infty}\cap K_{\infty}
\ar@{-}[rr]^{l^{\lambda+u}}\ar@{-}[d]_{l^{\tau-u}}&&K_{\infty}\\
k_{\infty}\ar@{-}[rru]_{l^{\tau+\lambda}}
}
\]
Because $\ge E\subseteq \cicl N{}$ is cyclotomic, the field
of constants of $E_{H^+}$ is also $\F$. 

As before, $\vartheta:=\sqrt[l^{\tau}]{\pi_{\infty}\mu}$
is a prime element of $E_{\infty}$. We have
\begin{gather*}
X^{l^{\tau}}-\pi_{\infty}\mu=\prod_{i=0}^{l^{\tau}-1}\big(X-\zeta_{
l^{\tau}}^i\vartheta\big).
\intertext{Hence}
\prod_{i=0}^{l^{\tau}-1}\big(-\zeta_{l^{\tau}}^i\vartheta
\big)=(-1)^{l^{\tau}}
\prod_{i=0}^{l^{\tau}-1}\big(\zeta_{l^{\tau}}^i\vartheta\big)=
(-1)^{l^{\tau}}\N_{E_{\infty}/k_{\infty}}\vartheta =
-\pi_{\infty}\mu.
\intertext{Thus}
\N_{E_{\infty}/k_{\infty}}\vartheta=(-1)^{l^{\tau}+1}\mu\pi_{\infty}.
\end{gather*}

Since the field of constants of $E_{H^+}$ is $\F$, then, from
Theorem \ref{TEx1.2}, there exists an element of degree $1$ in $\*{
E_{\infty}}$ satisfying
\[
\phi_{E_{\infty}}(X)=\phi_{\infty}(\N_{E_{\infty}/k_{\infty}} (X))=1.
\]
Set $X=\vartheta^s\alpha w$ with $s\in {\ma Z}, \alpha\in
\f, w\in U_{E_{\infty}}^{(1)}$. Since $\deg X=\deg \pi_{\infty}
v_{\infty}(X)=1\cdot s=s$, it follows that $s=1$.
Furthermore $\N_{E_{\infty}/k_{\infty}}
(w)\in U_{\infty}^{(1)}=\big(U_{\infty}^{(1)}\big)^{l^{\tau}}$ and
$\N_{E_{\infty}/k_{\infty}}(\alpha)=\alpha^{l^{\tau}}$. Therefore
\[
1=\phi_{E_{\infty}}(X)=\phi_{\infty}(\N_{E_{\infty}/k_{\infty}}
(X))=\phi_{\infty}\Big(\big((-1)^{l^{\tau}+1}\mu\pi_{\infty}\big)
\alpha^{l^{\tau}}u^{l^{\tau}}\Big)=(-1)^{l^{\tau}+1}\mu\alpha^{l^{\tau}},
\]
where $u\in U_{\infty}^{(1)}$. Therefore $-\mu=(-\alpha)^{-l^{
\tau}}\in \big(\f\big)^{l^{\tau}}$ and since $\tau<n$, it follows
that $-1\in \big(\f\big)^{l^{\tau}}$. Thus $\mu\in \big(\f\big)^{
l^{\tau}}$ and $E_{\infty}=
k_{\infty}\big(\sqrt[l^{\tau}]{\pi_{\infty}\mu}\big)=
k_{\infty}\big(\sqrt[l^{\tau}]{\pi_{\infty}}\big)$.

We obtain
\begin{gather*}
E_{\infty}K_{\infty}=K_{\infty}k_{\infty}
\big(\sqrt[l^{\tau}]{\pi_{\infty}}\big)=F\big(
\sqrt[l^{\tau}]{\pi_{\infty}\theta},\sqrt[l^{\tau}]{\pi_{\infty}}\big)=
F\big(\sqrt[l^{\tau}]{\pi_{\infty}\theta},\sqrt[l^{\tau}]
{\theta}\big)=
K_{\infty}\big(\sqrt[l^{\tau}]{\theta}\big).
\end{gather*}

The field of the constants of $EK$ is ${\ma F}_{q^{l^{\lambda+u}}}$,
therefore $[{\ma F}_{q^{l^{\lambda}}}\big(\sqrt[l^{\tau}]{\theta}\big):
{\ma F}_{q^{l^{\lambda}}}]=l^{u}$. From Theorem \ref{TEx1.6}
we have that $\theta\in \big(\*{{\ma F}_{q^{l^{\lambda}}}}\big)^{
l^{\tau-u}}\setminus \big(\*{{\ma F}_{q^{l^{\lambda}}}}\big)^{
l^{\tau-u+1}}$.

In short, $K_{\infty}=F\big(\sqrt[l^{\tau}]{\pi_{\infty}\theta }\big)$ with
$\theta\in \big(\*{{\ma F}_{q^{l^{\lambda}}}}\big)^{
l^{\tau-u}}\setminus \big(\*{{\ma F}_{q^{l^{\lambda}}}}\big)^{
l^{\tau-u+1}}$, 
$F=k_{\infty}{\ma F}_{q^{l^{\lambda}}}$,
and $\tilde\Pi=\sqrt[l^{\tau}]{\theta\pi_{\infty}}$.

The irreducible polynomial of $\tilde\Pi$ over $F$ is $X^{l^{\tau}}-
\theta\pi_{\infty}\in F[X]$. Then $X^{l^{\tau}}-\theta\pi_{\infty}=
\prod_{j=0}^{l^{\tau}-1}(X-\zeta^j_{l^{\tau}}\tilde\Pi)$ and
\begin{align*}
\N_{K_{\infty}/F}\tilde\Pi&=\prod_{j=0}^{l^{\tau}-1}\zeta_{l^{\tau}}^j
\tilde\Pi=(-1)^{l^{\tau}}\prod_{j=0}^{l^{\tau}-1}(-\zeta_{l^{\tau}}^j
\tilde\Pi)=(-1)^{l^{\tau}}(-\theta\pi_{\pi_{\infty}})=(-1)^{l^{\tau}+1}
\theta\pi_{\infty},\\
\N_{K_{\infty}/k_{\infty}}\tilde\Pi&=\N_{F/k_{\infty}}(
\N_{K_{\infty}/F}\tilde\Pi)=\N_{F/k_{\infty}}((-1)^{l^{\tau}+1}\theta
\pi_{\infty})\\
&=(-1)^{(l^{\tau}+1)l^{\lambda}}(\N_{F/k_{\infty}}\theta)
\pi_{\infty}^{l^{\lambda}}.
\end{align*}
Now $\N_{F/k_{\infty}}\*{\ma F}_{q^{l^{\lambda}}}=\f$.
Therefore, $\N_{F/k_{\infty}}\theta\in
\big(\*{{\ma F}_{q}}\big)^{
l^{\tau-u}}\setminus \big(\*{{\ma F}_{q}}\big)^{
l^{\tau-u+1}}$.

Let's see the norm of an arbitrary element $X$ of $\*K_{\infty}$.
Let $X=\tilde\Pi^s\Lambda\omega$ with $s\in{\ma Z}$, 
$\Lambda\in\*{\ma F}_{q^{l^{\lambda}}}$, and $\omega\in
U_{K_{\infty}}^{(1)}$. Then
\begin{gather*}
\begin{align*}
\N_{K_{\infty}/k_{\infty}}\omega&=\omega_0\in U_{\infty}^{(1)},\\
\N_{K_{\infty}/k_{\infty}}\tilde\Pi^s&=(-1)^{(l^{\tau}+1)l^{\lambda}s}
\xi^s\pi_{\infty}^{l^{\lambda}s}\quad \text{with}\quad \xi=
\N_{F/k_{\infty}}\theta\in\big(\*{{\ma F}_{q}}\big)^{
l^{\tau-u}}\setminus \big(\*{{\ma F}_{q}}\big)^{
l^{\tau-u+1}},\\
\N_{K_{\infty}/k_{\infty}}\Lambda&=\N_{F/k_{\infty}}
(\N_{K_{\infty}/F}\Lambda)=\N_{F/k_{\infty}}\Lambda^{l^{\tau}}=
(\N_{F/k_{\infty}}\Lambda)^{l^{\tau}},\\
&\text{and}\quad
\N_{K_{\infty}/k_{\infty}}\*{\ma F}_{q^{l^{\lambda}}}=(\f)^{l^{\tau}}.
\end{align*}
\intertext{Therefore}
\N_{K_{\infty}/k_{\infty}}X=(-1)^{(l^{\tau}+1)l^{\lambda}s}\xi^s
\pi_{\infty}^{l^{\lambda}s}\big(\N_{F/k_{\infty}}\Lambda\big)^{
l^{\tau}}\omega_0,
\intertext{and}
\begin{align*}
\phi_{K_{\infty}}(X)&=\phi_{\infty}(\N_{K_{\infty}/k_{\infty}}(X))=
(-1)^{(l^{\tau}+1)l^{\lambda}s}\xi^s\big(\N_{F/k_{\infty}}
\Lambda\big)^{l^{\tau}}\\
&=\big((-1)^{l^{\lambda}}\xi\big)^s
\big[(-1)^{l^{\lambda}s}(\N_{F/k_{\infty}}\Lambda)\big]^{l^{\tau}}.
\end{align*}
\intertext{Now}
X\in\ker\phi_{K_{\infty}}\iff \big((-1)^{l^{\lambda}}\xi\big)^s
\big[(-1)^{l^{\lambda}s}(\N_{F/k_{\infty}}\Lambda)\big]^{l^{\tau}}=1.
\end{gather*}

In other words, 
\[
X\in\ker\phi_{K_{\infty}}\iff \text{\ there exists $
\Lambda\in\*{\ma F}_{q^{l^{\lambda}}}$ such that 
$\big(\N_{F/k_{\infty}}\Lambda\big)^{l^{\tau}}
=\Big(\pm \xi^{-1}\Big)^s$}.
\]
Since $l^{\rho}|q-1$ with $\rho\geq 1$ and $l^n\nmid q-1$,
we have $n\geq 2$. Thus $-1\in (\f)^{l^{\tau-u+1}}$.
Therefore $\xi^{s}\in (\f)^{l^{\tau}}$. Since
$\xi\in\big(\*{{\ma F}_{q}}\big)^{
l^{\tau-u}}\setminus \big(\*{{\ma F}_{q}}\big)^{
l^{\tau-u+1}}$, it follows that $l^u|s$ and $l^u$
is the minimum positive integer
with this property. For such $X$, we have
\[
\deg_{K_{\infty}}X=\deg_{K_{\infty}}\pL_{\infty}v_{K_{\infty}}(X)=
l^{\lambda}\cdot l^{u}=l^{\lambda+u}.
\]

Therefore, the field of constants of $K_{H^+}$ is ${\ma F}_{l^{
\lambda+u}}$, the same of $\ge E K$.

We have obtained our first main result

\begin{theorem}\label{TExtendido5.19}
Let $K/k$ be a geometric cyclic extension of degree $l^n$ with $l$ a prime
number and $n\geq 1$. Then, if $E$ is given by 
{\rm{(\ref{EcExtended1})}}, we have
\begin{gather*}
\ge K=\ge E K. \tag*{$\fin$}
\end{gather*}
\end{theorem}

\section{General finite abelian extensions}\label{SEx3.2}

The main key to obtain the extended genus field of a finite
abelian extension, is the following:

\begin{lemma}\label{LEx2.1} Let $E_1$ and $E_2$ be two
finite cyclotomic extensions of $k$ and let $E=E_1E_2$. Then
$\ge E=\ge{(E_1)}\ge{(E_2)}$.
\end{lemma}

\begin{proof}
See \cite[Proposition 6.3]{BaMoReRzVi2018}.
\end{proof}

As a consequence we obtain our final main result.

\begin{theorem}\label{TEx2.2}
Let $K/k$ be any geometric finite abelian extension. Then if $E$
is given in {\rm{(\ref{EcExtended1})}}, we have that 
\[
\ge K=\ge E K.
\]
\end{theorem}

\begin{proof}
Let $K=K_1\cdots K_s$ where each $K_j/k$ is a cyclic
extension of prime power degree. Let $E=E_1\cdots E_s$
with each $E_j$ be given in (\ref{EcExtended1}). Then,
from Lemma \ref{LEx2.1}, we obtain
\[
\ge E=\ge{(E_1)}\cdots \ge{(E_s)}.
\]
Therefore, from 
Theorem \ref{TExtendido5.19}, it follows that 
\[
\ge K\subseteq \ge E K=\ge{(E_1)}K_1\cdots \ge{(E_s)}K_s=
\ge{(K_1)}\cdots \ge{(K_s)}\subseteq \ge K.
\]
Hence $\ge K=\ge E K$.
\end{proof}

We obtain some consequences from Theorem \ref{TEx2.2}. We consider
$K/k$ a geometric abelian finite extension. We have $\ge K=\ge E K$ and
$\g K=\g E^{\H} K$. Let $K\subseteq {_n\cicl N{}_m}$, ${\mathcal M}=
L_n k_m$ and $E={\mc M} K\cap \cicl N{}$. For any finite abelian
extension $L/J$ and any prime $\pK$ of $J$, we denote by $e^*_{\pK}
(L|J)$ to the tame ramification index of the prime $\pK$ in the extension
$L/J$, namely, if $e_{\pK}(L|J)=p^{\beta} \alpha$ with $\gcd(\alpha, p)=1$,
then $e^*_{\pK}(L|J)=\alpha$. The set of tame ramification indexes 
is multiplicative.

\begin{lemma}\label{LEx2.3}
We have $e_{\infty}(E|k)=e^*_{\infty}(K|k)$.
\end{lemma}

\begin{proof}
First note that if $k\subseteq J\subseteq {\mc M}$ then $e^*_{\infty}({\mc
M}|J)=e^*_{\infty}(J|k)=1$. Now we consider
\[
\xymatrix{
K\ar@{-}[rr]^{e^*_{\infty}=1}\ar@{-}[d]&&K{\mc M}\ar@{-}[d]\\
K\cap {\mc M}\ar@{-}[rr]^{e^*_{\infty}=1}&&{\mc M}
}
\]

Since $e^*_{\infty}(K{\mc M}|K)|e^*_{\infty}({\mc M}|K\cap{\mc M})$,
it follows that $e^*_{\infty}(K{\mc M}|K)=
e^*_{\infty}({\mc M}|K\cap{\mc M})=1$. Therefore $
e^*_{\infty}(K|K\cap{\mc M})=e^*_{\infty}(K{\mc M}|{\mc M})$.

We have
\[
\xymatrix{
E\ar@{-}[rrr]\ar@{-}[dd]&&&K{\mc M}\ar@{-}[dl]\ar@{-}[dd]\\
&&K\ar@{-}[dl]\ar@{-}[dll]\\ 
k\ar@{-}[r]&K\cap {\mc M}\ar@{-}[rr]&&{\mc M}
}
\qquad
\begin{minipage}{6cm}
\begin{align*}
e_{\infty}(K{\mc M}|k)&=e_{\infty}(K{\mc M}|{\mc M})e_{\infty}
({\mc M}|k)\\
&=e_{\infty}(K{\mc M}|E)e_{\infty}(E|k).
\end{align*}
\end{minipage}
\]

It follows
\begin{align*}
e^*_{\infty}(K{\mc M}|k)&=e^*_{\infty}(K{\mc M}|{\mc M})
=e^*_{\infty}(K|K\cap {\mc M})
=e^*_{\infty}(E|k)=e_{\infty}(E|k).
\intertext{Now}
e^*_{\infty}(K|K\cap {\mc M})&=e^*_{\infty}(K|K\cap {\mc M})
e^*_{\infty}(K\cap {\mc M}|k)
=e^*_{\infty}(K|k)=e_{\infty}(E|k).
\end{align*}
\end{proof}

\begin{theorem}\label{TEx2.4}
We have $[\ge K:\g K]=[\ge E:\g E^{\H}]=[\ge E:\g E]\cdot |\H|$.
In particular $[\ge K:\g K]|q-1$.
It also holds 
\[
f_{\infty}(\ge K:\g K)=|\H|,\quad e_{\infty}(\ge K:\g K)=[\ge E:\g E].
\]

Furthermore, the field of constants of $\g K$ is ${\ma F}_{q^{\deg_K\p}}$,
$\deg_K\p=f_{\infty}(K|k)$ and the field of constants of both, $\ge K$
and $K_{H^+}$ is ${\ma F}_{q^{|\H|\deg_K\p}}$ and we have
$|\H|\deg_K\p= f_{\infty}(EK|k)$.
\end{theorem}

\begin{proof}
We have 
\[
\xymatrix{
E\ar@{-}[r]\ar@{-}[d]&EK\ar@{-}[d]\\
E^{\H}\ar@{-}[r]\ar@{-}[d]&E^{\H}K\ar@{-}[d]\\
E\cap K\ar@{-}[r]&K
}
\qquad\qquad
\xymatrix{
\ge E\ar@{-}[r]\ar@{-}[d]&\ge K=\ge E K\ar@{-}[d]\\
\g E\ar@{-}[d]\ar@{-}[r]&\g E K\ar@{-}[d]\\
\g E^{\H}\ar@{-}[r]&\g K=\g E^{\H} K
}
\]
It follows that $[\ge K:\g K]=[\ge E:\g E^{\H}]|q-1$.

Now, since $\g E K/(\g E K)^{\H}=\g K$ is an extension of constants
of degree $|\H|$ (\cite[Theorem 2.2]{BaMoReRzVi2018}),
in fact, $|\H|=f_{\infty}(\g EK|\g K)$, we will see that the extension
$\ge K/\g EK$ is totally ramified.
\[
\xymatrix{
\ge E\ar@{-}[rr]^{e^*_{\infty}=1}\ar@{-}[d]_{e_{\infty}=
[\ge E:\g E]}&& \ge E K\ar@{-}[d]\\
\g E\ar@{-}[rr]^{e^*_{\infty}=1}\ar@{-}[dd]&&
\g EK\ar@{-}[dl]^{e_{\infty}=1}\\
&K\ar@{-}[ld]\\k
}
\]

We have $e_{\infty}(\g EK|k)=e_{\infty}(K|k)$. Hence $e^*_{
\infty}(\g EK|\g E)=1$. 
Similarly, we obtain $e^*_{\infty}(\ge EK|\ge E)=1$.

Therefore $e_{\infty}(\ge E K|\g EK)=e_{\infty}(\ge E|\g E)=
[\ge E:\g E]$ and $\ge K/\g EK$ is totally ramified.

Since $\deg_k\p=1$ and $K/k$ is geometric, we obtain that
$f_{\infty}(K|k)=\deg_K\p$. We know that the field of constants
of $\g K$ is ${\ma F}_{q^{\deg_K\p}}$.

Finally, $\g EK/\g E^{\H}K$ is an extension of constants of
degree $|\H|=f_{\infty}(EK|K)$. Hence, the field of constants
of both, $\ge K$ and $K_{H^+}$, is ${\ma F}_{q^{\deg_K\p\cdot
|\H|}}={\ma F}_{q^{f_{\infty}(EK|k)}}$.
\end{proof}

\end{document}